\documentclass{scrartcl}

\usepackage{lmodern}
\usepackage[utf8]{inputenc}
\usepackage[T1]{fontenc}
\usepackage[english]{babel}
\usepackage{csquotes}

\usepackage[backend=biber,style=alphabetic,eprint=false,url=false,giveninits=true,sorting=anyt,date=year,maxbibnames=6]{biblatex}
\usepackage{amsmath,amssymb,amsthm}
\usepackage{mathtools,thmtools,stmaryrd,esint}
\usepackage{mathrsfs}
\usepackage{bm} % fold bold math

\usepackage{hyperref}
\usepackage[shortlabels]{enumitem}
\usepackage[nameinlink]{cleveref}

\usepackage{longtable,tabu}
\hypersetup{colorlinks=true,breaklinks=true}
\numberwithin{equation}{section}

\newcommand{\N}{\mathbb{N}}

\newcommand{\R}{\mathbb{R}}

\newcommand{\probspace}{\mathscr{P}}
\newcommand{\contspace}{\mathscr{C}}
\newcommand{\ener}{E}

\newcommand{\mres}{\mathbin{\vrule height 1.6ex depth 0pt width
0.13ex\vrule height 0.13ex depth 0pt width 1.3ex}}
\newcommand{\eps}{\varepsilon}
\newcommand{\hdm}{{\mathscr H}}
\newcommand{\lbm}{{\mathscr L}}

\newcommand{\dd}{\mathop{}\mathopen{}\mathrm{d}}
\newcommand{\entropy}{\mathrm{Ent}}
\newcommand{\Opt}{\mathsf{Opt}}
\newcommand{\mbf}[1]{\bm{#1}}
\newcommand{\OT}{\mathrm{OT}}
\newcommand{\spt}[1]{S_{#1}}
\newcommand{\loc}{\mathrm{loc}}

\DeclarePairedDelimiter\abs{\lvert}{\rvert}
\DeclarePairedDelimiter\norm{\lVert}{\rVert}

\DeclareMathOperator{\diam}{diam}
\DeclareMathOperator{\tr}{Tr}

\DeclareMathOperator{\dist}{dist}
\DeclareMathOperator{\id}{Id}
\DeclareMathOperator*{\argmin}{arg\,min}

\newcommand{\xto}[2][]{\xrightarrow[#1]{#2}}

\newcommand{\weakto}{\rightharpoonup}

\declaretheorem[name=Theorem,within=section]{theorem}
\declaretheorem[name=Lem\-ma,numberlike=theorem]{lemma}
\declaretheorem[name=Proposition,numberlike=theorem]{proposition}
\declaretheorem[name=Corollary,numberlike=theorem]{corollary}
\declaretheorem[name=Definition,numberlike=theorem,style=definition]{definition}
\declaretheorem[name=Remark,numberlike=theorem,style=remark]{remark}
\declaretheorem[name=Example,numberlike=theorem,style=remark]{example}

\begin{document}
\title{Convergence rate of general entropic optimal transport costs}
\author{Guillaume Carlier
	\thanks{CEREMADE, Université Paris-Dauphine, Université PSL, CNRS, Mokaplan, Inria Paris, 75016 Paris, France.
		email: carlier@ceremade.dauphine.fr}
	\and Paul Pegon
	\thanks{CEREMADE, Université Paris-Dauphine, Université PSL, CNRS, Mokaplan, Inria Paris, 75016 Paris, France.
		email: pegon@ceremade.dauphine.fr}
	\and Luca Tamanini
	\thanks{Department of Decision Sciences and BIDSA, Bocconi University, Via Roberto Sarfatti 25, 20100 Milano MI, Italy.
		email: luca.tamanini@unibocconi.it}
}

%\date\today
\date{}
\maketitle

\begin{abstract}
We investigate the convergence rate of the optimal entropic cost $v_\varepsilon$ to the optimal transport cost as the noise parameter $\varepsilon \downarrow 0$. We show that for a large class of cost functions $c$ on $\mathbb{R}^d\times \mathbb{R}^d$ (for which optimal plans are not necessarily unique or induced by a transport map) and compactly supported and $L^{\infty}$ marginals, one has $v_\varepsilon-v_0= \frac{d}{2} \varepsilon \log(1/\varepsilon)+ O(\varepsilon)$. Upper bounds are obtained by a block approximation strategy and an integral variant of Alexandrov's theorem. Under an infinitesimal twist condition on $c$, i.e.\ invertibility of $\nabla_{xy}^2 c(x,y)$, we get the lower bound by establishing a quadratic detachment of the duality gap in $d$ dimensions thanks to Minty's trick.

%We study the convergence, as the noise parameter $\eps \downarrow 0$, of the cost $v_\eps$ of a large class of entropic optimal transport problems. More precisely, under an infinitesimal twist condition on $c$, i.e.\ invertibility of $\nabla_{xy}^2 c(x,y)$, we obtain an (almost) first-order Taylor expansion for $v_\eps$ around $\eps=0$.

%We also investigate the regularity of $\eps \mapsto v_\eps$ for $\eps>0$, proving that it is analytic.

\end{abstract}

\vskip\baselineskip\noindent
%49Q22 Optimal transportation
%94A17 Measures of information, entropy
%49K40 Sensitivity, stability, well-posedness
%49N15 Duality theory (optimization)

% ArXiv categories.  Optimization and Control (math.OC), Analysis of PDEs (math.AP), Functional Analysis (math.FA)

\tableofcontents

%%%%%%%%%%%%%%%%%%%%%%%%%%%%%%%%%%%%%%%%%%%%%%%%%%%%%%

\section*{Notations}
\begin{longtabu}  {X[2,c,m]  X[10,l,p]}
$x,\mbf{x}$ & generic points of $\R^d$ and of $\R^d\times \R^d$ respectively;\\
$\abs{\cdot}$ & Euclidean norm on $\R^d$;\\
$\norm{\cdot}$ & norm on $\R^d \times \R^d$ defined by $\norm{\mbf x} = \max \{\abs{x},\abs{y}\}$ if $\mbf x = (x,y)$;\\
$B_r(x),B_r(\mbf x)$ & open ball of radius $r$ centered at $x\in\R^d$ or $\mbf x\in (\R^d)^2$ for the above norms;\\
$\contspace^{0,1}(X)$ & space of real-valued Lipschitz functions on $X$ which is a subset of $\R^d$ or $(\R^d)^2$;\\
$[f]_{\contspace^{0,1}(X)}$ & Lipschitz constant of $f : X \to \R$ where $X$ is a subset of $\R^d$ or $(\R^d)^2$ for the above norms;\\
$\contspace_\loc^{1,1}(\Omega)$ & space of differentiable real-valued functions on $\Omega$, an open subset of $\R^d$ or $(\R^d)^2$, with locally Lipschitz gradient;\\
$\probspace(\R^d)$ & space of probability measures on $\R^d$;\\
$\spt{\mu}$ & support of the measure $\mu$;\\
$M_d(\R)$ & space of real matrices of size $d \times d$, endowed with the Frobenius norm induced by the scalar product $A \cdot B \coloneqq\tr(A^T B)$, for $A,B\in M_d(\R)$;\\
$S_d(\R)$ & subspace of real symmetric matrices of size $d \times d$.
\end{longtabu}

\section{Introduction}

We consider two probability measures $\mu^\pm$ compactly supported in $X^\pm \subseteq \R^d$ and a cost function $c : X^-\times X^+ \to \R$. The \emph{Entropic Optimal Transport problem} (also called \emph{entropy-regularized optimal transport problem}, EOT for short) reads as:
\begin{equation}\label{eot-intro}
v_\eps \coloneqq \inf\Bigg\{\int_{X^-\times X^+} c(x,y)\,\dd\gamma(x,y) + \eps\entropy(\gamma\,|\,\mu^- \otimes \mu^+)\Bigg\}
\end{equation} 
where the infimum is taken among all \emph{couplings} $\gamma$ between $\mu^-$ and $\mu^+$, i.e. probability measures having $\mu^-$ and $\mu^+$ as marginals. The classical optimal transport (OT) problem corresponds to $\eps = 0$. The penalization term $\entropy$ is the Boltzmann-Shannon relative entropy (also called Kullback-Leibler divergence) and $\eps>0$ can be interpreted as a temperature parameter. Heuristically, this consists in moving $\mu^-$ onto $\mu^+$ in the cheapest and (at the same time) most ``diffuse'' way, since the (unique) minimizer $\gamma_\eps$ of \labelcref{eot-intro} is forced to be absolutely continuous with respect to $\mu^- \otimes \mu^+$ because of the entropy term and so its mass has to be ``spread out'', in contrast with solutions to the unperturbed transport problem.

%As $\eps \downarrow 0$ we can easily conjecture that the spreading effect vanishes and thus EOT converges to OT.

In the last decade, this class of problems has witnessed a rapidly increasing interest and is now an extremely active research topic, because it has found numerous applications and proved to be an efficient way to approximate OT problems, especially from a computational viewpoint. Indeed, when it comes to solving EOT by alternating Kullback-Leibler projections on the two marginal constraints, by the algebraic properties of the entropy such iterative projections correspond to the celebrated Sinkhorn's algorithm \cite{sinkhornRelationshipArbitraryPositive1964}, applied in this framework in the pioneering works \cite{cuturiSinkhornDistancesLightspeed2013, benamouIterativeBregmanProjections2015}. The simplicity and the good convergence guarantees (see \cite{franklinScalingMultidimensionalMatrices1989, marinoOptimalTransportApproach2020}) of this method compared to the algorithms used for the Monge-Kantorovich problem, then determined the success of EOT for applications in machine learning, statistics, image processing, language processing and other areas (see the monograph \cite{peyreComputationalOptimalTransport2019} and references therein).

\bigskip

As it appears clearly from \labelcref{eot-intro}, EOT is a perturbed transport problem and, as such, it is natural to investigate its behaviour as the parameter $\eps$ vanishes. In this direction, several aspects deserve to be studied, such as the convergence of optimal values, potentials and optimal plans, possibly with quantitative rates. Leveraging on a large deviations interpretation and on the notion of $(c,\eps)$-cyclical monotonicity, the two last questions have been addressed very recently in \cite{berntonEntropicOptimalTransport2022, nutzEntropicOptimalTransport2021}.

As concerns the convergence of the optimal values (also called \emph{entropic costs}) denoted by $v_\eps$, let us mention earlier contributions. In the pioneering works \cite{mikamiMongeProblemQuadratic2004, mikamiOptimalTransportationProblem2008, leonardSchrodingerProblemMonge2012} (which tackled the question from the Schr\"odinger problem's viewpoint) and \cite{carlierConvergenceEntropicSchemes2017}, the $\Gamma$-convergence of the EOT problem towards the unregularized OT problem was proved in the quadratic case, i.e.\ for $c(x,y) = |x-y|^2$. As a consequence, the optimal value of the quadratic EOT problem converges to the optimal value of the quadratic OT problem, namely the squared Wasserstein distance.

Since then, this convergence result was generalized in two directions: at the level of the accuracy and at the level of the cost function. As for the former, in \cite{adamsLargeDeviationsPrincipleWasserstein2011, erbarLargeDeviationsWasserstein2015, palDifferenceEntropicCost2019} the first-order asymptotic expansion for the cost is established, both in a pointwise and $\Gamma$-convergence sense. A further improvement has been obtained independently in \cite{confortiFormulaTimeDerivative2021, chizatFasterWassersteinDistance2020}, where also the second-order term in the expansion is determined (under a regularity assumption on the Wasserstein geodesic connecting the two marginals). A second-order expansion in the same spirit has been obtained for semi-discrete OT in \cite{altschulerAsymptoticsSemidiscreteEntropic2022}. As regards the second direction (more general costs), the $\Gamma$-convergence result actually holds true for a very large class of continuous cost functions, but the first-order expansion is much more difficult to extend. In \cite{palDifferenceEntropicCost2019}, this is achieved for costs satisfying strong regularity assumptions (roughly speaking, a uniform ellipticity condition on the Hessian of the Kantorovich potential of the associated OT problem).

\medskip

The aim of this paper is to continue in the second direction, namely to weaken even further the assumptions on the cost function $c$. First of all, it is folklore that $v_\eps$ is a $\contspace^\infty$ function of $\eps > 0$ (and even analytic as we shall see in \Cref{thm:analytic}). However, the differentiability of $v_\eps$ at $\eps = 0^+$ is equivalent to the existence of a finite entropy solution for $v_0$, which is generally false, apart from the discrete case. Instead, one expects $v_\eps-v_0$ to be of the order of $\eps \log(1/\eps)$. We shall prove that it is indeed the case and that the remainder is $O(\eps)$ under quite general assumptions, one of which, inspired by \cite{maRegularityPotentialFunctions2005, mccannRectifiabilityOptimalTransportation2012},  is the \emph{infinitesimal twist condition} which requires the cross-derivative $\nabla^2_{xy}c(x,y)$ to be invertible for every $(x,y)$. It is worth stressing that for costs satisfying this condition, the solutions to the associated OT problem need not be concentrated on a graph and may even fail to be unique, see \cite[Example 3.1]{mccannRectifiabilityOptimalTransportation2012}. Our main findings may be summarized as:

\begin{theorem}\label{thm:main}
Let $\Omega^\pm \subseteq \R^d$ be open convex sets and $\mu^\pm$ be absolutely continuous probability measures compactly supported on $\Omega^\pm$ and with $L^\infty$ densities. Assume that the cost $c$ is $\contspace^2$ on $\Omega^- \times \Omega^+$ and infinitesimally twisted. Then
\[
v_\eps = v_0 + \frac{d}{2}\eps \log(1/\eps)+ O(\eps),
\]
$v_\eps$ being the value of \labelcref{eot-intro}.
\end{theorem}

Hence $v_\eps$ approximates the transport cost $v_0$ with accuracy $\frac{d}{2}\eps\log(1/\eps)$ and the next term in the expansion is at most of order $\eps$ ; it cannot be better in general, as shown in the cases handled in \cite{palDifferenceEntropicCost2019}. As an application, if we debias the problem, namely if we consider the Sinkhorn divergence (see \cite{ramdasWassersteinTwoSampleTesting2017})
\[
\OT_\eps(\mu^-,\mu^+) \coloneqq v_\eps(\mu^-,\mu^+) - \frac{1}{2}\big(v_\eps(\mu^-,\mu^-) + v_\eps(\mu^+,\mu^+)\big)
\]
instead of the entropic cost $v_\eps$, it immediately follows from our main theorem that the term in $\eps\log(1/\eps)$ disappears and therefore
\[
\OT_\eps(\mu^-,\mu^+) = v_0(\mu^-,\mu^+) + O(\eps).
\]
For the proof of Theorem \ref{thm:main} we will show separately that for $\eps \ll 1$,
\begin{equation}\label{eq:aim}
v_\eps \leq v_0 + \frac{d}{2}\eps\log(1/\eps) + M\eps \qquad \textrm{and} \qquad v_\eps \geq v_0 + \frac{d}{2}\eps\log(1/\eps) - m\eps,
\end{equation}
for some constants $m,M$. 

For the \textbf{upper bound}, the proof relies on the block approximation introduced in \cite{carlierConvergenceEntropicSchemes2017} and on an integral variant of Alexandrov's theorem tailored to our purpose (see \Cref{lem:alexandrov-modified}). Let us emphasize that this upper bound only requires $\contspace^{1,1}$ regularity and no twist condition. We also treat more general situations where the cost function is only Lipschitz, in which case the constant $d/2$ has to be replaced by the \emph{upper entropy dimension} (without the $1/2$ factor) of the marginals (see \Cref{value_general_upper_bound}), and this is shown to be sharp in  \Cref{ex:sharpness}.

For the \textbf{lower bound}, from the dual formulation of \labelcref{eot-intro} it follows:
\[
v_\eps \geq v_0 - \eps\log\int_{\R^d \times \R^d}e^{-\frac{E(x,y)}{\eps}}\,\dd\mu^-(x)\dd\mu^+(y),
\]
where $E(x,y) \coloneqq c(x,y) - \phi(x) - \psi(y)$ is the duality gap and $(\phi,\psi)$ are Kantorovich potentials for the OT problem. By using the so-called Minty's trick as presented in \cite{mccannRectifiabilityOptimalTransportation2012} (here the infinitesimal twist condition on $c$ is required), we are able to prove that $E$ detaches quadratically from the set $\{E=0\}$ and this allows us to estimate the previous integral in the desired way. We also explain in the quadratic case how to derive a quantitative stability result for optimal plans from the same trick.

\subsection*{Related literature}

The discrete case with finitely supported marginals falls within the class of finite-dimensional linear programming problems. In this framework, the choice of the entropy as regularizing function has a long tradition and a detailed investigation of convergence rates for regularized primal/dual formulations and expansion of the cost was already addressed in the early work \cite{cominettiAsymptoticAnalysisExponential1994}. 

In the continuous setting, entropic regularization is a more recent technique \cite{carlierConvergenceEntropicSchemes2017} and has been initially employed in the quadratic case $c(x,y) = |x-y|^2$, where EOT appears as a perturbation of the (squared) Wasserstein distance. For this choice of cost function, EOT is intimately linked (and actually equivalent in the Euclidean setting) to a much older problem, known as Schr\"odinger problem (SP for short, see the survey \cite{leonardSurveySchrodingerProblem2014}). The latter can be better described as a control problem for Brownian particles, hence a dynamical problem, unlike EOT which is static in nature. This dynamical aspect turns out to be very useful to overcome regularity problems affecting Wasserstein geodesics, as already shown in \cite{gigliSecondOrderDifferentiation2021, gentilEntropicInterpolationProof2020}. The fact that, for $c(x,y) = |x-y|^2$, EOT and SP coincide (see \cite{gigliBenamouBrenierDuality2020}) explains why the quadratic EOT problem has attracted a lot of attention coming from different research areas: not only the ones mentioned before, where EOT has proved to be a valuable tool, but also (stochastic) control, statistical mechanics and many others.

In the very last years, several generalizations of both EOT and SP have appeared. As concerns the latter, extensions to mean-field and interacting particle systems \cite{backhoffMeanFieldSchrodinger2020, chiariniSchrodingerProblemLattice2021}, and to an abstract framework \cite{monsaingeonDynamicalSchrodingerProblem2020} are worth mentioning. As regards EOT, regularizing functions other than the entropy (and their impact on algorithms) have been considered for instance in \cite{dimarinoOptimalTransportLosses2020, lorenzOrliczSpaceRegularization2022}. Moreover, adding entropy penalizations in the spirit of EOT has also been used in the study of multi-marginal OT \cite{carlierLinearConvergenceMultimarginal2022, benamouGeneralizedIncompressibleFlows2019, marinoOptimalTransportApproach2020}, incompressible fluids \cite{arnaudonEntropicInterpolationProblem2020, baradatSmallNoiseLimit2020} and unbalanced OT \cite{baradatRegularizedUnbalancedOptimal2021}, where the entropy term gives rise to a branching unbalanced OT problem.

For an account of latest developments let us finally mention a series of recent works \cite{ghosalStabilityEntropicOptimal2021, ecksteinQuantitativeStabilityRegularized2021, nutzStabilitySchrodingerPotentials2022}, where stability with respect to marginals of plans and potentials for EOT is addressed.

\subsection*{Structure of the paper}

In \Cref{sec:reg}, we recall some main features of the OT and EOT problems and prove the analyticity of $v_\eps$ (\Cref{thm:analytic}). The upper and lower bounds on $v_\eps$ are established respectively in \Cref{sec:upper} (\Cref{value_general_upper_bound} and \Cref{value_general_upper_bound2}) and \Cref{sec:lower} (\Cref{prop:lower-bound}). %In the former we prove the first in \labelcref{eq:aim} under only $\contspace^{1,1}$-regularity of the cost (i.e.\ with no infinitesimal twist condition), see , whereas in the latter the second inequality in \labelcref{eq:aim}, see .

\section{Regularity of the EOT cost}\label{sec:reg}

Let us first collect all notations, definitions and relevant results concerning the optimal transport problem and its entropy-regularized counterpart. As already anticipated in the introduction, given a parameter $\eps>0$, two measures $\mu^\pm \in \probspace(\R^d)$ with compact support $X^\pm \coloneqq \spt{\mu^\pm}$ and a continuous cost function $c \in \contspace(X^-\times X^+)$, the EOT problem is given by
\begin{equation}\tag{EOT$_\eps$}\label{eot_pb}
v_\eps \coloneqq \inf_{\gamma \in \Pi(\mu^-,\;\mu^+)}\Bigg\{\int_{X^-\times X^+} c(x,y)\,\dd\gamma(x,y) + \eps\entropy(\gamma\,|\,\mu^- \otimes \mu^+)\Bigg\},
\end{equation}
where $\Pi(\mu^-,\;\mu^+)$ denotes the set of couplings between $\mu^-$ and $\mu^+$ and $\entropy(\cdot\,|\,\mu^- \otimes \mu^+)$ the Boltzmann-Shannon relative entropy (or Kullback-Leibler divergence) w.r.t.\ the product measure $\mu^- \otimes \mu^+$, defined for general probability measures $p,q$ as
\[
\entropy(p \,|\, q) = 
\begin{dcases*}
\displaystyle{\int_{\R^d} \rho \log(\rho)\, \dd q} & if $p = \rho q$,\\
+\infty & otherwise.
\end{dcases*}
\]
The fact that $q$ is a probability measure ensures that $\entropy(p \,|\, q) \geq 0$. The value of \labelcref{eot_pb}, denoted by $v_\eps$, is called \emph{entropic cost}. 

The dual problem of \labelcref{eot_pb} in the sense of convex analysis reads as
\[
\sup_{\substack{\phi \in \contspace(X^-)\\\psi \in \contspace(X^+)}}\Bigg\{\int_{X^-}\phi\,\dd\mu^- + \int_{X^+}\psi\,\dd\mu^+ - \eps\int_{X^- \times X^+}e^{\frac{\phi(x)+\psi(y)-c(x,y)}{\eps}}\,\dd\mu^-(x)\dd\mu^+(y) + \eps \Bigg\},
\]
which is invariant by $(\phi,\psi) \mapsto (\phi+\lambda,\psi-\lambda)$ where $\lambda \in\R$. Another way to write the dual problem, which is also invariant by $\phi \mapsto \phi + \lambda$ and $\psi \mapsto \psi + \lambda$, reads as
\begin{equation}\tag{D$_\eps$}\label{dual_eot_pb}
\sup_{\substack{\phi \in \contspace(X^-)\\\psi \in \contspace(X^+)}}\Bigg\{\int_{X^-}\phi\,\dd\mu^- + \int_{X^+}\psi\,\dd\mu^+ - \eps\log\int_{X^- \times X^+}e^{\frac{\phi(x)+\psi(y)-c(x,y)}{\eps}}\,\dd\mu^-(x)\dd\mu^+(y)\Bigg\},
\end{equation}
see \cite{leonardMinimizationEnergyFunctionals2001} or \cite{nutzIntroductionEntropicOptimal2022} for a more recent presentation. From \labelcref{eot_pb} and \labelcref{dual_eot_pb} we recover, as $\eps \to 0$, the (unregularized) optimal transport problem associated with $c$ and its dual, that we recall for the reader's sake:
\begin{equation}\tag{OT}\label{ot_pb}
v_0 \coloneqq \inf_{\gamma \in \Pi(\mu^-,\;\mu^+)}\int_{X^- \times X^+} c(x,y)\,\dd\gamma(x,y)
\end{equation}
and
\begin{equation}\tag{D}\label{dual_ot_pb}
\sup_{\substack{\phi \in  \contspace(X^-)\\\psi \in  \contspace(X^+)}}\Bigg\{\int_{X^-}\phi\,\dd\mu^- + \int_{X^+}\psi\,\dd\mu^+ : \phi \oplus \psi \leq c\Bigg\},
\end{equation}
respectively, where $\phi \oplus \psi(x,y) \coloneqq \phi(x)+\psi(y)$. The set of optimal couplings between $\mu^-$ and $\mu^+$ will be denoted by $\Opt(\mu^-,\mu^+)$. The value of \labelcref{ot_pb}, denoted by $v_0$, is called \emph{transport cost}. As for \labelcref{dual_ot_pb}, it is well known that optimizers, called \emph{Kantorovich potentials}, exist whenever $X^\pm$ are compact and they can be chosen $c$-conjugate (see \cite[Chapter~1]{santambrogioOptimalTransportApplied2015}):
\[\phi = \inf_{y\in X^+} c(\cdot,y)-\psi(y),\quad\psi = \inf_{x\in X^-} c(x,\cdot)-\phi(x).\]

The link between EOT and OT is very strong, as already discussed in the introduction. For later use, recall that a consequence of the $\Gamma$-convergence of \labelcref{eot_pb} towards \labelcref{ot_pb} is that
\begin{equation}\label{eq:small-time}
\lim_{\eps \downarrow 0}v_\eps = v_0.
\end{equation}
A proof tailored to our setting can be found in \cite[Section 2.3]{carlierConvergenceEntropicSchemes2017} (although formulated for the quadratic cost, it actually works for any continuous cost as long as $\mu^\pm$ have compact supports).

\medskip

The fact that \labelcref{eot_pb} admits a unique solution $\gamma_\eps$, called \emph{optimal entropic plan}, is a consequence of the direct method in the calculus of variations and strict convexity of the entropy. The structure of $\gamma_\eps$ is then very rigid: indeed, there exist two real-valued Borel functions $\phi_\eps,\psi_\eps$ such that
\begin{equation}\label{eq:structure}
\gamma_\eps = \exp\Bigg(\frac{\phi_\eps \oplus \psi_\eps - c}{\eps}\Bigg) \mu^- \otimes \mu^+,
\end{equation}
which implies in particular that
\begin{equation}\label{eq:veps-rewritten}
v_\eps = \int_{X^-}\phi_\eps\,\dd\mu^- + \int_{X^+}\psi_\eps\,\dd\mu^+
\end{equation}
and these functions, which have continuous representatives, are a.s.\ uniquely determined up to additive constants, in the sense that if \labelcref{eq:structure} holds for $(\phi,\psi)$ and $(\phi',\psi')$, then $\phi=\phi'+\lambda$ $\mu^-$-a.e.\ and $\psi=\psi'-\lambda$ $\mu^+$-a.e.\ for some $\lambda \in \R$. Moreover, $\gamma_\eps$ is the unique coupling $\gamma \in \Pi(\mu^-,\;\mu^+)$ whose density w.r.t.\ $\mu^- \otimes \mu^+$ can be written as in \labelcref{eq:structure}. A self-contained proof of all these facts in the case of compactly supported measures $\mu^\pm$ can be found in \cite[Proposition 2.1]{gigliSecondOrderDifferentiation2021}. The reader is refered to the analysis of \cite{marinoOptimalTransportApproach2020}, to the notes of Nutz \cite{nutzIntroductionEntropicOptimal2022} for a more general framework, and to \cite{borweinDecompositionMultivariateFunctions1992, borweinEntropyMinimizationDAD1994, csiszarDivergenceGeometryProbability1975, follmerEntropyMinimizationSchrodinger1997, ruschendorfClosednessSumSpaces1998} for earlier references.

The functions $\phi_\eps,\psi_\eps$ in \labelcref{eq:structure} are called \emph{Schr\"odinger potentials}, the terminology being motivated by the fact that they solve the dual problem \labelcref{dual_eot_pb}. Furthermore, $\phi_\eps,\psi_\eps$ are the (unique) solutions to the so-called \emph{Schr\"odinger system}
\begin{equation}\label{eq:Ssystem}
\left\{\begin{array}{ll}
\displaystyle{\phi(x) = -\eps\log\int_{X^+}e^{\frac{\psi(y)-c(x,y)}{\eps}}\,\dd\mu^+(y)} & \qquad \textrm{for }\mu^-\textrm{-a.e. }x,\\
\displaystyle{\psi(y) = -\eps\log\int_{X^-}e^{\frac{\phi(x)-c(x,y)}{\eps}}\,\dd\mu^-(x)} & \qquad \textrm{for }\mu^+\textrm{-a.e. }y,
\end{array}\right.
\end{equation}
named after E.\ Schr\"odinger who introduced it in the seminal papers \cite{schrodingererwinUberUmkehrungNaturgesetze1931, schrodingerTheorieRelativisteElectron1932} (see also \cite{chetriteSchrodinger1931Paper2021} for a recent English translation of the former). Note that \labelcref{eq:Ssystem} is a softmin version of the classical $c$-conjugacy relation for Kantorovich potentials. From \labelcref{eq:Ssystem} and the continuity of $c$, we easily deduce that $\phi,\psi$ share a common modulus of continuity with $c$ (see \cite{marinoOptimalTransportApproach2020} for details). From an OT viewpoint, \labelcref{eq:Ssystem} simply means that $e^{(\phi \oplus \psi -c)/\eps}\mu^- \otimes \mu^+ \in \Pi(\mu^-,\;\mu^+)$.

We are now ready to state and prove the first main result. Recall that $v_\eps$ denotes the value of \labelcref{eot_pb} and $v_0$ the one of \labelcref{ot_pb}; the dependence on $\mu^+,\mu^- \in \probspace(\R^d)$ is omitted, as they will always be fixed.

\begin{theorem}\label{thm:analytic}
Let $\mu^\pm \in \probspace(\R^d)$ have compact support $X^\pm \subseteq \R^d$ and $c\in   \contspace(X^-\times X^+)$. Then the function $\varepsilon \mapsto v_\varepsilon$ is continuous, non-decreasing and concave on $[0,+\infty)$ ; it is analytic on $(0,+\infty)$.
\end{theorem}

\begin{proof}
The convergence of $v_\eps$ to $v_0$ as $\eps \downarrow 0$ is precisely \labelcref{eq:small-time}. The fact that $v$ is non-decreasing follows from the non-negativity of the entropy; concavity is also obvious since $v_\eps$ is an infimum of affine functions of $\eps$.

Analyticity of $v$ on $(0,+\infty)$ can be deduced from the analytic version of the implicit function theorem, adapting the arguments of \cite{carlierDifferentialApproachMultiMarginal2020} as follows. Let $\eps>0$ and, recalling that Schr\"odinger potentials are unique up to the trivial transformation $(\phi, \psi) \mapsto (\phi-\lambda, \psi+\lambda)$ with $\lambda\in \R$, impose the following normalization condition
\begin{equation}\label{normaliz}
\int_{X^-} \phi \dd \mu^-=0.
\end{equation}
Denote then by $(\phi_\eps,\psi_\eps)$ the unique functions satisfying \labelcref{eq:Ssystem} and \labelcref{normaliz} and note that by \labelcref{eq:veps-rewritten} and \labelcref{normaliz} the entropic cost writes as
\begin{equation}\label{exprvpsi}
v_\eps=\int_{X^+} \psi_\eps \dd \mu^+.
\end{equation}
As already noticed, $\phi_\eps$ and $\psi_\eps$ are continuous in $X^-$ and $X^+$ respectively, whence the validity of both equations in \labelcref{eq:Ssystem} in a pointwise sense. This means that, if we define
\[
\begin{split}
F(\eps, \phi, \psi)(x) & \coloneqq e^{\frac{\phi(x)}{\eps}} \int_{X^+} e^{\frac{-c(x,y)+\psi(y)}{\eps}} \dd \mu^+(y), \qquad x \in X^- \\
G(\eps, \phi, \psi)(y) & \coloneqq e^{\frac{\psi(y)}{\eps}} \int_{X^-} e^{\frac{-c(x,y)+\phi(x)}{\eps}} \dd \mu^-(x), \qquad y \in X^+
\end{split}
\]
the Schr\"odinger system \labelcref{eq:Ssystem} is equivalent to
\[
F(\eps, \phi_\eps, \psi_\eps)(x) = 1 \quad (\forall x\in X^-), \qquad \textrm{and} \qquad G(\eps, \phi_\eps, \psi_\eps)(y) = 1 \quad (\forall y\in X^+).
\]
Finally, define the spaces
\[
A \coloneqq \{(\phi, \psi) \in  \contspace(X^-)\times \contspace(X^+) \; : \; \mbox{ (\ref{normaliz}) holds}\}
\]
and
\[
B \coloneqq \Bigg\{(f, g) \in  \contspace(X^-)\times \contspace(X^+) \; :   \int_{X^-} f \dd \mu^-=\int_{X^+} g \dd \mu^+\Bigg\},
\]
which are Banach spaces when equipped with the uniform norm, and the map $S : (0, +\infty)\times \contspace(X^-)\times \contspace(X^+)\to \contspace(X^-)\times \contspace(X^+)$
\[
S(\eps, \phi, \psi) \coloneqq (F(\eps, \phi, \psi), G(\eps, \phi, \psi)).
\]
Let us observe that $S$ is analytic, takes values in $B$ and the pair of normalized Schr\"odinger potentials $(\phi_\eps, \psi_\eps)$ is implicitly defined by
\[
(\phi_\eps, \psi_\eps)\in A \quad \textrm{and} \quad S(\eps, \phi_\eps, \psi_\eps)=(1,1).
\]
We claim that for every $\eps>0$ and $(\phi, \psi)\in A$, the derivative of $S$ w.r.t.\ $(\phi,\psi)$ at $(\eps, \psi, \phi)$ is an isomorphism between $A$ and $B$, analyticity of $\eps\mapsto (\phi_\eps, \psi_\eps)$ will then follow from the implicit function theorem in the analytic case. To ease notations assume $\eps=1$ and let $(\phi, \psi)\in A$ be fixed. Defining 
\[
S=S(1,\cdot, \cdot) = (F(1,\cdot, \cdot), G(1,\cdot , \cdot)), \qquad \tilde S \coloneqq (\log (F), \log(G)),
\]
and
\[
L \coloneqq S'(\phi, \psi) = (L_1, L_2), \qquad \tilde L \coloneqq \tilde S'(\phi, \psi) = (\tilde L_1, \tilde L_2), 
\]
for $(h, k)\in \contspace(X^-) \times \contspace(X^+)$ we then have
\[ 
\tilde L_1(h,k)(x)=h(x) + \frac{\displaystyle{\int_{X^+} e^{-c(x,y)+ \psi(y)} k(y) \dd \mu^+(y)}}{\displaystyle{\int_{X^+} e^{-c(x,y)+ \psi(y)} \dd \mu^+(y)}}, \quad L_1(h, k) = F(\phi, \psi)  \tilde L_1(h,k),\]
and
\[
\tilde L_2(h,k)(y)=k(y) + \frac{\displaystyle{\int_{X^-} e^{-c(x,y)+ \phi(x)} h(x) \dd \mu^-(x)}}{\displaystyle{\int_{X^-} e^{-c(x,y)+ \phi(x)} \dd \mu^-(x)}}, \quad  L_2(h, k) = G(\phi, \psi)  \tilde L_2(h,k). 
\]
We can write in a more synthetic way $\tilde L$ as
\begin{equation}\label{tlcondit}
\tilde L_1(h,k)(x)=h(x)+ \int_{X^+} k(y) \dd Q^x(y), \quad \tilde L_2(h,k)(y)=k(y)+ \int_{X^-} h(x) \dd Q_y(x),
\end{equation}
where $Q^x$ and $Q_y$ are the disintegration measures of the probability measure
\[\dd Q (x,y)\coloneqq\frac{e^{\phi(x)+\psi (y)-c(x,y)} \dd \mu^{-}(x)  \dd \mu^{+}(y)}{\displaystyle{\int_{X^-\times X^+} e^{\phi(x')+\psi (y')-c(x',y')} \dd \mu^{-}(x') \dd \mu^{+}(y')}}
\]
with respect to its first marginal $\alpha$ and second marginal $\beta$, that is:
\begin{equation}\label{disintQ}
Q =\alpha \otimes Q^x= Q_y \otimes \beta
\end{equation}
(in other words given $(X,Y)$ with law $Q$, $Q^x$ is the conditional law of $Y$ given $X=x$ and $Q_y$ is the conditional law of $X$ given $Y=y$). Assume now that $(h,k) \in \contspace(X^-) \times \contspace(X^+)$ belongs to $\ker(L)=\ker(\tilde L)$, i.e.
\begin{equation}\label{hkinkernel1}
h(x)+ \int_{X^+} k(y) \dd Q^x(y)=0, \qquad \forall x\in X^-=\spt{\alpha}
\end{equation}
and
\begin{equation}\label{hkinkernel2}
k(y)+ \int_{X^-} h(x) \dd Q_y(x)=0, \qquad \forall y\in X^+=\spt{\beta}
\end{equation}
where $X^-=\spt{\alpha}$ and $X^+=\spt{\beta}$ follow from the equivalence of $Q$ with $\mu^- \otimes \mu^+$ (in the sense that they have the same negligible sets). Multiplying \labelcref{hkinkernel1} by $h(x)$ and integrating w.r.t.\ $\alpha$, multiplying \labelcref{hkinkernel2} by $k(y)$ and integrating w.r.t.\ $\beta$, and summing the two, by \labelcref{disintQ} we get
\[ 
\int_{X^-\times X^+} (h(x)+k(y))^2 \dd Q(x,y) = 0
\]
from which we deduce that $h(x)+ k(y)=0$ for $\mu^{-} \otimes \mu^+$-a.e.\ $(x,y)$, that is $h$ and $k$ are two constants that sum to $0$. We thus have $\ker(\tilde L) = \ker(L) = \{(\lambda, -\lambda) \,:\, \lambda \in \R\}$ 
and in particular $\ker(L) \cap A=\{0\}$, so that $L$ is injective on $A$. Since $c$ is uniformly continuous on $X^- \times X^+ $, all functions in the image of $\tilde L$ are equi-continuous and equi-bounded, and it follows from Ascoli's Theorem that $\tilde L$ is a compact perturbation of the identity map on $\contspace(X^-) \times \contspace(X^+)$. This enables us to invoke the Fredholm alternative theorem to deduce from the fact that  $\ker(\tilde L)$ is a line that $\tilde L(\contspace(X^-)\times \contspace(X^+))=\tilde L(A)$ is a hyperplane of $\contspace(X^-)\times \contspace(X^+)$. Since $\tilde L(A)$  is isomorphic to $L(A)$, and $L(A)$ is included in the hyperplane $B$ (because the original non-linear map $S$ takes values in $B$), we have $L(A)=B$. 

Hence $L$ is an isomorphism between $A$ and $B$ and, as anticipated, the implicit function theorem for analytic maps (see chapter 4 of \cite{buffoniAnalyticTheoryGlobal2016}) therefore implies that $(\phi_\eps, \psi_\eps)$ depends analytically on $\eps>0$. In particular $(0,+\infty) \ni \eps \mapsto \psi_\eps \in \contspace(X^-)$ is analytic too, whence the analyticity of $v$ on $(0, +\infty)$ thanks to \labelcref{exprvpsi}. 
\end{proof}

\begin{remark}
The entropic cost $\eps \mapsto v_\eps$ being analytic on $(0,+\infty)$, it can be extended in an analytic fashion to some complex values of $\eps$. Note  that the proof above shows that the entropic cost is also an analytic function of the transport cost $c$.
\end{remark}

\begin{remark}\label{rmk:general-setting}
Even if Theorem \ref{thm:analytic} is stated in the Euclidean setting, it is worth stressing that it is in fact true in a much more general framework. Indeed, it is not difficult to see that the previous proof works verbatim over compact metric spaces and an even higher level of generality is possible. More precisely, not even a topological structure is needed. It is indeed sufficient that $c \in L^\infty(\mu^- \otimes \mu^+)$ with $\mu^-,\mu^+$ probability measures on arbitrary measurable spaces, as it is the case in \cite{carlierDifferentialApproachMultiMarginal2020}. However, with no topological structure the proof of the analyticity of $\eps \mapsto v_\eps$ would be more involved.
\end{remark}

Knowing that all order derivatives of $\eps \mapsto v_\eps$ exist on $(0,+\infty)$, it is natural to look for an explicit expression for them. The case of the first derivative is particularly easy to handle, as shown below. 

\begin{corollary}\label{cor:taylor}
Under the same assumptions of Theorem \ref{thm:analytic}, it holds
\begin{equation}\label{eq:first-der}
v_\eps' = \entropy(\gamma_\eps\,|\,\mu^- \otimes \mu^+), \qquad \forall \eps > 0.
\end{equation}
Besides, if there exists $\gamma_0 \in \Opt(\mu^- \otimes \mu^+)$ with $\entropy(\gamma_0\,|\,\mu^- \otimes \mu^+) < +\infty$, then $\eps \mapsto v_\eps$ belongs to $\contspace^1([0,+\infty))$ with right derivative at $\eps=0$ given by
\begin{equation}\label{eq:right-der0}
v_\eps'|_{\eps=0} = \inf_{\gamma \in \Opt(\mu^-,\mu^+)}\entropy(\gamma\,|\,\mu^- \otimes \mu^+),
\end{equation}
in other words:
\begin{equation}\label{eq:taylor}
v_\eps - v_0 = \eps\inf_{\gamma \in \Opt(\mu^-,\mu^+)}\entropy(\gamma\,|\,\mu^- \otimes \mu^+) + o(\eps).
\end{equation}
Moreover, in that case, $\gamma_\eps$ converges narrowly to the unique coupling $\gamma^\ast$ of minimal entropy in $\Opt(\mu^-,\mu^+)$, namely
\[
\gamma^\ast = \argmin_{\gamma \in \mathrm{Opt}(\mu^-,\mu^+)} \entropy(\gamma\,|\,\mu^- \otimes \mu^+).
\]
\end{corollary}

\begin{proof}
Knowing that $\eps \mapsto v_\eps$ is analytic ($\contspace^1$ would be enough), \labelcref{eq:first-der} is a consequence of the envelope theorem and of the uniqueness of the minimizer of \labelcref{eot_pb}.

As concerns the second part of the statement, define
\[
C_\eps(\gamma) \coloneqq \int_{X^- \times X^+} c\,\dd\gamma + \eps\entropy(\gamma\,|\,\mu^- \otimes \mu^+)
\]
and observe that $v_\eps-v_0 = C_\eps(\gamma_\eps)-C_0(\gamma) \leq C_\eps(\gamma)-C_0(\gamma) = \eps\entropy(\gamma\,|\,\mu^- \otimes \mu^+)$ for any $\gamma \in \Opt(\mu^-,\mu^+)$, hence
\begin{equation}\label{eq:limsup}
\limsup_{\eps \downarrow 0}\frac{v_\eps-v_0}{\eps} \leq \inf_{\gamma \in \Opt(\mu^-,\mu^+)}\entropy(\gamma\,|\,\mu^- \otimes \mu^+),
\end{equation}
the right-hand side being finite thanks to the existence of $\gamma_0$ as in the statement. For the liminf inequality, recall that \labelcref{eot_pb} $\Gamma$-converges to \labelcref{ot_pb} as $\eps \to 0$, so that up to subsequences $\gamma_\eps$ converges narrowly to $\gamma^* \in \Opt(\mu^-,\mu^+)$. Combining this information with $v_\eps-v_0 \geq C_\eps(\gamma_\eps)-C_0(\gamma_\eps) = \eps\entropy(\gamma_\eps\,|\,\mu^- \otimes \mu^+)$ and the lower semicontinuity of the entropy (see \cite[\S{7.1.2}]{santambrogioOptimalTransportApplied2015}) yields
\[
\liminf_{\eps \downarrow 0}\frac{v_\eps-v_0}{\eps} \geq \entropy(\gamma^*\,|\,\mu^- \otimes \mu^+) \geq \inf_{\gamma \in \Opt(\mu^-,\mu^+)}\entropy(\gamma\,|\,\mu^- \otimes \mu^+),
\]
and this inequality\footnote{By the way, this also shows that the differentiability of $v_\eps$ at $\eps = 0$ is in fact equivalent to the existence of an optimal plan with finite entropy.} together with \labelcref{eq:limsup} implies the right differentiability of $v_\eps$ at $\eps=0$ as well as \labelcref{eq:right-der0}. 

Note that the argument above also proves the existence of a $\gamma^* \in \Opt(\mu^-,\mu^+)$ of minimal entropy towards which $\gamma_\eps$ converges narrowly (up to subsequences) as well as the lower semicontinuity of $v_\eps'$ at $\eps=0$. The uniqueness of such $\gamma^*$ comes from the convexity of $\Opt(\mu^-,\mu^+)$ and the strict convexity of the entropy, hence the convergence of the whole sequence. The upper semicontinuity of $v_\eps'$ at $\eps=0$ is instead a consequence of concavity, so that in conclusion $v_\eps'$ is continuous up to $\eps=0$, namely $v_\eps \in \contspace^1([0,+\infty))$.

% Finally, the validity of \labelcref{eq:taylor} is a consequence of the $\contspace^1$-regularity up to $\eps=0$ of $v_\eps$, which allows for a first-order Taylor expansion. The coefficients are then easily identified since we know who $v_0'$ is.
\end{proof}

Let us stress that while \labelcref{eq:first-der} is an immediate consequence of Theorem \ref{thm:analytic}, the proof of \labelcref{eq:right-der0} (inspired by \cite{confortiFormulaTimeDerivative2021, monsaingeonDynamicalSchrodingerProblem2020}) requires the existence of an optimal coupling with finite entropy w.r.t.\ $\mu^- \otimes \mu^+$ and the reader may wonder whether this condition is reasonable or not. If $\mu^-$ and $\mu^+$ are finitely-supported measures, the condition is trivially satisfied; hence identities \labelcref{eq:right-der0} and \labelcref{eq:taylor} may be of some interest in this case (in this direction, recall the pioneering work \cite{cominettiAsymptoticAnalysisExponential1994}).

However, as soon as $\mu^-$ and $\mu^+$ have densities, the aforementioned condition is not satisfied in general. In particular, it cannot hold whenever optimal plans are concentrated on $d$-rectifiable sets. To obtain a non-trivial asymptotic expansion of the entropic cost around $\eps=0$ in more general situations, a more accurate study of $v_\eps-v_0$ is thus required, and this will be carried out in Sections \ref{sec:upper} and \ref{sec:lower}.

\section{Upper bound}\label{sec:upper}

Our goal in this section is to establish upper bounds on $v_\eps -v_0$ up to order $O(\eps)$ as $\eps \to 0$. We start with a general upper bound holding for $\contspace^{0,1}$ (i.e. Lipschitz) costs, then we prove a finer upper bound for $\contspace^{1,1}$ costs. The leading order terms, which are proportional to $\eps \log(1/\eps)$, will be shown to be sharp in this section and in \Cref{sec:lower}.  

\subsection{Upper bound for \texorpdfstring{$\contspace^{0,1}$}{C0,1} costs}
The upper bound we are going to establish will depend on the dimension of $\mu^-$ and $\mu^+$.  The natural notion of dimension which arises in our context is the \emph{entropy dimension}, also called \emph{information dimension} or \emph{R\'enyi dimension}, since it was originally introduced by R\'enyi in \cite{renyiDimensionEntropyProbability1959}. The definition that we use corresponds to the one given in \cite{youngDimensionEntropyLyapunov1982}: if $\mu$ is a probability measure over $\R^d$, we set for every  $\delta > 0$,
\[H_\delta(\mu) = \inf \left\{ \sum_{n\in\N} \mu(A_n) \log(1/\mu(A_n)) : \forall n, \diam(A_n) \leq \delta, \text{ and } \R^d = \bigcup_{n\in \N} A_n\right\},\]
where the infimum is taken over countable coverings $(A_n)_{n\in\N}$ of $\R^d$ by Borel subsets, and we define the \emph{lower and upper entropy dimension of $\mu$} respectively by:
\[\underline \dim_E(\mu) \coloneqq \liminf_{\delta \to 0^+} \frac{H_\delta(\mu)}{\log(1/\delta)}, \quad \overline \dim_E(\mu) \coloneqq \limsup_{\delta \to 0^+} \frac{H_\delta(\mu)}{\log(1/\delta)}.\]

When $\mu$ has compact support, notice that its upper entropy dimension is always smaller than the upper box dimension of the support of $\mu$, that is:
\[\overline \dim_E(\mu) \leq \overline \dim_B(\spt{\mu}) \coloneqq \limsup_{\delta\to 0^+} \frac{\log(N_\delta(\spt{\mu}))}{\log(1/\delta)},\]
where for every $A\subseteq \R^d$, $N_\delta(A)$ denotes the box-counting number of $A$ at scale $\delta$, i.e. the minimal number of sets of diameter $\delta > 0$ which cover $A$. Indeed if $(A_1, \ldots, A_{N_\delta(\spt{\mu})})$ is such a covering, by concavity of $t\mapsto t\log(1/t)$, we have
\begin{equation}\label{upper_bound_box_counting_number}
\sum_{1\leq n \leq N_\delta(\spt{\mu})} \mu(A_n) \log(1/\mu(A_n)) \leq \log N_\delta(\spt{\mu}).
\end{equation}
The reader interested in the different notions of dimension of sets and measures may consult \cite{falconerTechniquesFractalGeometry1997}.

\begin{proposition}\label{value_general_upper_bound}
If $\mu^\pm \in \probspace(\R^d)$ and $c \in \contspace^{0,1}(\spt{\mu^-}\times \spt{\mu^+})$, then for every $\eps>0$,
\begin{equation}\label{general_upper_bound_precise}
v_\eps \leq v_0 +\eps (H_\eps(\mu^-) \wedge H_\eps(\mu^+))+  [c]_{\contspace^{0,1}(\spt{\mu^-}\times\spt{\mu^+})} \eps.
\end{equation}
In particular, setting $d^\pm = \overline\dim_E(\mu^\pm)$,
\begin{equation}\label{general_upper_bound_entropy_dimension}
v_\eps \leq v_0 + (d^+\wedge d^-) \eps \log(1/\eps) + o(\eps\log(1/\eps)) \quad \text{as $\eps \to 0$},
\end{equation}
and if $\mu^\pm$ are concentrated on compact submanifolds of dimension $d^\pm$,
\begin{equation}\label{general_upper_bound_manifold_dimension}
v_\eps \leq v_0 + (d^+ \wedge d^-)\eps \log(1/\eps) + O(\eps).
\end{equation}
\end{proposition}
The proof is based on a generalization of the block approximation defined in \cite[Definition~2.9]{carlierConvergenceEntropicSchemes2017}.
\begin{proof}
 Let $\gamma_0 \in \Opt(\mu^-,\mu^+)$. First of all, let us extend the cost to a function $c \in \contspace^{0,1}(\R^d\times \R^d)$ with same Lipschitz constant $L \coloneqq [c]_{\contspace^{0,1}(\R^d\times \R^d)} = [c]_{\contspace^{0,1}(\spt{\mu^-}\times \spt{\mu^+})}$. We consider a covering $\R^d = \bigcup_{n\in\N} A_n$ of Borel sets such that $\diam(A_n) \leq \delta$ for every $n\in\N$. We set for every $n\in \N$,
\[\mu_n^\pm \coloneqq\begin{dcases*}
\frac {\mu^\pm \mres A_n}{\mu^\pm(A_n)}& if $\mu^\pm(A_n) > 0$,\\
0& otherwise,
\end{dcases*}\]
then for every pair $i,j \in \N$, 
\[(\gamma_0)_{i,j} \coloneqq \gamma_0(A_i \times A_j) \mu_i^- \otimes \mu_j^+,\]
and finally,
\[\gamma^\delta \coloneqq \sum_{i,j\in \N} (\gamma_0)_{i,j}.\]
By definition, $\gamma^\delta \ll \mu^- \otimes \mu^+$ and we may check that its marginals are $\mu^\pm$, for instance:
\[(\pi_1)_\#\gamma^\delta = \sum_{i,j\in\N} (\pi_1)_\#  (\gamma_0(A_i \times A_j) \mu_i^- \otimes \mu_j^+) = \sum_{i,j\in\N}\gamma_0(A_i\times A_j) \mu_i^- = \sum_{i\in\N} \mu^-(A_i) \mu_i^- = \mu^-.\] 
Besides, for $\mu^-\otimes\mu^+$-almost every $(x,y) \in A_i\times A_j$,
\[\frac{\dd \gamma^\delta}{\dd\mu^-\otimes \mu^+}(x,y) = \frac{\gamma_0(A_i \times A_j)}{\mu^-(A_i) \mu^+(A_j)} \quad \text{if } \mu^-(A_i) \mu^+(A_j) > 0,\]
and it is $0$ otherwise. Let us compute its entropy, summing only over indices $i,j$ such that $\mu^-(A_i) > 0$ and $\mu^+(A_j) > 0$.
\begin{align*}
\entropy(\gamma^\delta \,|\, \mu^- \otimes \mu^+) &= \sum_{i,j\in \N} \int_{A_i \times A_j} \log\left(\frac{\gamma_0(A_i \times A_j)}{\mu^-(A_i) \mu^+(A_j)}\right) \dd \gamma^\delta\\
&= \sum_{i,j\in\N} \gamma_0(A_i \times A_j) \log\left(\frac{\gamma_0(A_i \times A_j)}{\mu^-(A_i) \mu^+(A_j)}\right)\\
&= \sum_{i,j\in\N} \gamma_0(A_i\times A_j) \log\left(\frac{\gamma_0(A_i\times A_j)}{\mu^-(A_i)}\right)+ \sum_{j\in\N} \mu^+(A_j) \log(1/\mu^+(A_j))\\
& \leq \sum_{j\in\N} \mu^+(A_j) \log(1/\mu^+(A_j)),
\end{align*}
the last inequality coming from the inequality $\gamma_0(A_i\times A_j) \leq \mu^-(A_i)$. Taking arbitrary coverings $(A_n)_{n\in\N}$ of diameter smaller than $\delta$, we get by definition of $H_\delta$,
\[\entropy(\gamma^\delta \,|\, \mu^- \otimes \mu^+)\leq H_\delta(\mu^+).\]
Notice that $W_\infty(\gamma^\delta,\gamma_0) \leq \delta$, where $W_p$ denotes the $p$-Wasserstein distance\footnote{For measures on $\R^d\times \R^d$, we consider the Wasserstein distance with respect to the norm that we introduced, namely $\norm{(x,y)} = \max \{\abs{x},\abs{y}\}$ for every $(x,y) \in \R^d\times \R^d$.} between compactly supported probability measures for $p\in [1,\infty]$.
As a consequence, taking $\gamma^\delta$ as competitor in \labelcref{eot_pb}, since $W_1\leq W_\infty$ we obtain
\begin{equation}\label{upper_bound_eps_delta}
\begin{aligned}
v_\eps \leq \int_{\R^d\times\R^d} c \dd\gamma^\delta + \eps H_\delta(\mu^+) &= v_0 + \int_{\R^d\times\R^d} c \dd(\gamma^\delta - \gamma_0) + \eps H_\delta(\mu^+)\\
&\leq v_0 + L W_1(\gamma^\delta,\gamma_0) + \eps H_\delta(\mu^+)\\
&\leq v_0 + L \delta  + \eps H_\delta(\mu^+).
\end{aligned}
\end{equation}
Taking $\delta = \eps$ and interchanging the roles of $\mu^-$ and $\mu^+$ yields \labelcref{general_upper_bound_precise}, and \labelcref{general_upper_bound_entropy_dimension} follows from the definition of the upper entropy dimension. Finally, if $\mu^\pm$ are concentrated on compact $d^\pm$-dimensional manifolds $M^\pm$, knowing that their box counting number $N_\eps(M^\pm)$ is bounded from above by $C/\eps^{d^\pm}$ for some constant $C > 0$ and every $\eps \in(0,1)$, applying \labelcref{upper_bound_box_counting_number} to $\mu^\pm$ yields $H_\eps(\mu^\pm) \leq d^\pm\log(1/\eps) + \log C$, thus
\[v_\eps \leq v_0 + (d^-\wedge d^+) \eps\log(1/\eps) + ( L + \log C)\eps\]
and \labelcref{general_upper_bound_manifold_dimension} is proved.
\end{proof}

\begin{remark}\label{rmk:unif-continuous}
If the cost $c$ is only assumed to be uniformly continuous on $\spt{\mu^-}\times\spt{\mu^+}$, with modulus of continuity $\omega_c : \delta \mapsto \sup_{\norm{\mbf x-\mbf x'}\leq \delta} \abs{c(\mbf x)-c(\mbf x')}$, a straightforward adaptation of the above proof shows that
\[v_\eps \leq v_0 + (d^+\wedge d^- + o(1)) \eps \log(1/\omega_c^{-1}(\eps)),\]
and if $\mu^\pm$ are concentrated on $d^\pm$-dimensional submanifolds,
\[v_\eps \leq v_0 + (d^+\wedge d^-) \eps \log(1/\omega_c^{-1}(\eps)) + O(\eps),\]
where $\omega_c^{-1}$ is the generalized inverse defined by $\omega_c^{-1}(t) \coloneqq \inf \{ s\geq 0 : \omega_c(s) \geq t\}$.
\end{remark}

The following easy example shows that the leading order term is in general sharp in the class of $\contspace^{0,1}$ cost functions.

\begin{example}\label{ex:sharpness}
Consider $\mu^- = \mu^+ = \lbm \mres[0,1]$ and $c(x,y) \coloneqq \abs{x-y}$ for $x,y\in [0,1]$. Then
\[v_\eps \geq v_0 + \eps \log(1/\eps).\]
Indeed, the unique optimizer to \labelcref{ot_pb} is $\gamma_0 = (\id,\id)_\# \mu^-$ with cost $v_0 = 0$, and the pair $(\phi_0,\psi_0) = (0,0)$ is a solution to the dual problem \labelcref{dual_ot_pb}. We use $(\phi_0,\psi_0)$ as a competitor in \labelcref{dual_eot_pb} to get
\[v_\eps \geq -\eps \log\left(\int_{[0,1]^2} e^{-\frac{\abs{y-x}}\eps} \dd x \dd y\right).\]
We have
\begin{align*}
\int_{[0,1]^2} e^{-\frac{\abs{y-x}}\eps} \dd x \dd y = 2 \int_0^1 \int_0^y e^{-\frac t\eps} \dd t \dd y&=2 \int_0^1 \eps (1-e^{-\frac y\eps}) \dd y =2 \eps(1-\eps(1-e^{-\frac 1\eps})) \leq 2 \eps,
\end{align*}
thus for every $\eps > 0$,
\[v_\eps \geq -\eps \log 2\eps= v_0 +\eps\log(1/\eps)+O(\eps),\]
and together with \labelcref{general_upper_bound_manifold_dimension}, since $d^\pm = 1$, we get $v_\eps-v_0 = \eps\log(1/\eps) + O(\eps)$.
\end{example}

\subsection{Upper bound for \texorpdfstring{$\contspace^{1,1}$}{C1,1} costs and \texorpdfstring{$L^\infty$}{Linfty} marginals}\label{sec:fine_upper_bound}

Building upon the same block approximation used to establish \Cref{value_general_upper_bound}, we provide a finer upper bound when the cost is of class $\contspace^{1,1}$ and $\mu^\pm$ have $L^\infty$ densities.

\begin{proposition}\label{value_general_upper_bound2}
Let $c\in \contspace_\loc^{1,1}(\Omega^- \times \Omega^+)$ where $\Omega^\pm \subseteq \R^d$ are open convex sets and $\mu^\pm \in L^\infty(\Omega^\pm)$ be two probability measures compactly supported in $\Omega^\pm$. Then there exists a constant $M\geq 0$ such that for every $\eps \in (0,1)$,
\begin{equation}\label{fine_upper_bound}
v_\eps \leq v_0 +\frac d2 \eps\log(1/\eps)+  M \eps.
\end{equation}
\end{proposition}
\begin{remark}
Notice that \Cref{value_general_upper_bound} would only show that $v_\eps \leq v_0 +d\eps\log(1/\eps) + O(\eps)$. The term $d\eps\log(1/\eps)$ term is generally sharp as shown in \Cref{ex:sharpness}, when $c$ is only Lipschitz, but \Cref{value_general_upper_bound2} shows that $d$ may actually be replaced by $d/2$ when $c$ is $\contspace^{1,1}$. This is in turn sharp in the class of $\contspace^{1,1}$ costs, as already exhibited in the quadratic case in \cite{adamsLargeDeviationsPrincipleWasserstein2011} or in the setting of \cite{palDifferenceEntropicCost2019}, but it is also a consequence of our lower bound stated in \Cref{prop:lower-bound} in the case of infinitesimally twisted cost functions.
\end{remark}

To prove \Cref{value_general_upper_bound2}, we will need a quadratic bound on the average error between a $\lambda$-convex function $f$ and its first-order Taylor expansion. We recall that a function $f : \Omega \to \R$ defined on a convex set $\Omega\subseteq \R^d$ is $\lambda$-convex, for $\lambda \in \R$, if $\displaystyle x\mapsto f(x) - \lambda \frac{\abs{x}^2}2$ is convex, and $f$ is $\lambda$-concave if $-f$ is $(-\lambda)$-convex. If $\Omega$ is open and $f$ is $\lambda$-convex, then it is differentiable almost everywhere on $\Omega$ by Rademacher's Theorem, and its distributional Hessian $D^2 f$ is a $S_{d}(\R)$-valued Radon measure over $\Omega$, which may be written (see \cite[\S{1.3}]{ambrosioFunctionsBoundedVariation2000}) as $D^2 f = \sigma_f \abs{D^2 f}$ where $\abs{D^2 f}$ is a positive Radon measure, and $\sigma_f : \Omega \to S_d(\R)$ is a Borel unit matrix field for the Frobenius norm.

\begin{lemma}\label{lem:alexandrov-modified}
Let $f : \Omega \to \R$ be a $\lambda$-convex function on a convex open set for some $\lambda \leq 0$. There exists a constant $C \geq 0$ depending only on $d$ such that:
\begin{equation}\label{integral_lambda_convexity_Taylor_inequality}
\int_{\Omega} \sup_{y\in B_r(x)\cap \Omega} \abs{f(y) - (f(x)+ \nabla f(x) \cdot (y-x))} \dd x \leq C r^2 \hdm^{d-1}(\partial \Omega) ([f]_{\contspace^{0,1}(\Omega)} + \abs{\lambda} \diam(\Omega)).
\end{equation}
\end{lemma}

The proof is based on that of Alexandrov's Theorem, asserting the existence of a second-order Taylor expansion at almost every point, which is given in \cite[\S{6.4}]{evansMeasureTheoryFine2015}.

\begin{proof}
\textbf{Step~1.} Let $g$ be a function which is convex on $B_r(x)$ and assume that $x$ is a differentiability point of $g$ and a Lebesgue point of both $g$ and $\nabla g$. Let us show that
\begin{equation}\label{lambda_convexity_Taylor_inequality}
\sup_{y\in B_{r/2}(x)} \abs{g(y) - (g(x)+\nabla g(x) \cdot (y-x))} \leq C r^2 \frac{\abs{D^2 g}(B_r(x))}{\lbm^d(B_r(x)))}
\end{equation}
for a constant $C= C(d)$. For $\eps \in (0,r)$, we set $r_\eps = r-\eps$ and we define $g_\eps \coloneqq g \star \eta_\eps$ on $B_{r_\eps}(x)$, where 
$\eta_\eps$ is a standard mollifier supported on $B_\eps$. Since $g_\eps \in \contspace^2(B_{r_\eps}(x))$, by the classical Taylor integral formula, for every $h\in B_{r_\eps}$,
\[g_\eps(x+h) = g_\eps(x) + \nabla g_\eps(x) \cdot h + \int_0^1 (1-t) D^2 g_\eps(x+th) \cdot h \otimes h \dd t,\]
and integrating w.r.t.\ $h$ over $B_{r_\eps}$, we get the bound:
\begin{align*}
\int_{B_{r_\eps}} \abs{g_\eps(x+h)-(g_\eps(x)+\nabla g_\eps(x) \cdot h)} \dd h &\leq r_\eps^2 \int_0^1 (1-t)\abs{D^2 g_\eps}(B_{r_\eps t}(x)) \dd t\\
&\leq r_\eps^2 \int_0^1 (1-t)\abs{D^2 g}(B_{r_\eps t+\eps}(x)) \dd t\\
&\leq \frac{r^2}2 \abs{D^2 g}(B_{r}(x)).
\end{align*}
The second inequality is a simple consequence of the identity $\abs{\mu}(O) = \sup \{ \int \phi \cdot \dd \mu : \phi\in \contspace_c(O,S_d(\R)), \norm{\phi}\leq 1\}$ when $\mu$ is a $S_d(\R)$-valued measure and $O$ is open. Besides, since $x$ is assumed to be a differentiability point of $g$, and a Lebesgue point of $g$ and $\nabla g$, we may pass to the inferior limit on the left-hand side and use Fatou's lemma to get:
\begin{equation}\label{convexity_Taylor_inequality}
\int_{B_r} \abs{g(x+h)-(g(x)+\nabla g(x) \cdot h)} \dd h \leq \frac{r^2}2 \abs{D^2 g}(B_r(x)).
\end{equation}
We set $T_x u(y) = u(x) + \nabla u(x) \cdot (y-x)$ whenever a function $u$ is differentiable at $x$. By convexity of $g$ on $B_r(x)$, \cite[Theorem~6.7~(ii)]{evansMeasureTheoryFine2015} yields\footnote{Notice that the proof works as soon as $g$ is defined on $B_r(x)$, not necessarily on the whole space $\R^d$.}
\begin{equation}\label{convexity_sup_inequality}
\sup_{B_{r/2}(x)} \abs{g-T_x g} \leq C \fint_{B_r(x)} \abs{g-T_x g},
\end{equation}
for a constant $C = C(d)$, which together with \labelcref{convexity_Taylor_inequality} gives \labelcref{lambda_convexity_Taylor_inequality}. 

\textbf{Step~2.} Given a convex function $g$ over $\Omega$, let us show that:
\begin{equation}\label{convexity_sup_integral_inequality}
\int_{\Omega} \sup_{B_{r/2}(x)} \abs{g-T_x g}\dd x \leq C \hdm^{d-1}(\partial \Omega) r^2 [g]_{\contspace^{0,1}(\Omega)},
\end{equation}
for some constant $C = C(d)$. Without loss of generality, we may assume that $\Omega$ is bounded and $[g]_{\contspace^{0,1}(\Omega)} < \infty$, otherwise the right-hand side is trivial. Since $g$ is convex, Lebesgue-almost every point $x \in \Omega$ is a differentiability point of $g$ by Rademacher's Theorem, and a Lebesgue point of $g$ and $\nabla g$ by Lebesgue's Theorem, thus integrating \labelcref{lambda_convexity_Taylor_inequality} over $\Omega_r \coloneqq \{ x : \dist(x, \R^d\setminus \Omega) > r \}$ and using Fubini-Tonelli's Theorem,
\begin{equation}\label{bound_average_convex1}
\begin{aligned}
\int_{\Omega_r} \sup_{B_{r/2}(x)} \abs{g-T_x g} \dd x &\leq Cr^2 \int_{\Omega} \int_{B_r(y)\cap \Omega_r} \frac{1}{\omega_{d} r^{d}} \dd x \dd\abs{D^2 g}(y)\\
&\leq C r^2 \abs{D^2 g}(\Omega).
\end{aligned}
\end{equation}
Again, we consider the regularization $g_\eps = g\star \eta_\eps$ over $\Omega_{\eps}$ for small $\eps >0$. By convexity of $g_\eps$, we know that $\abs{D^2 g_\eps} \leq C \Delta g_\eps$ on $\Omega_{\eps}$ for some constant $C = C(d)$. By Stokes' Theorem,
\[\int_{\Omega_{\eps}} \Delta g_\eps = -\int_{\partial \Omega_{\eps}} \nabla g_\eps \cdot \nu \leq \hdm^{d-1}(\partial \Omega_{\eps})[g_\eps]_{\contspace^{0,1}(\Omega_{\eps})} \leq \hdm^{d-1}(\partial \Omega) [g]_{\contspace^{0,1}(\Omega)},\]
where we have used the inequality $\hdm^{d-1}(\partial \Omega_{\eps}) \leq \hdm^{d-1}(\partial \Omega)$, which holds because $\partial \Omega_{\eps} = p(\partial \Omega)$ where $p$ is the projection onto the convex set $\overline{\Omega}_{\eps}$. Taking the limit $\eps \to 0$ yields, by weak convergence $D^2 g_\eps \mres \Omega_\eps \weakto D^2 g$ in $\contspace_c(\Omega,S_d(\R))'$,
\[\abs{D^2 g}(\Omega) \leq \liminf_{\eps\to 0} \int_{\Omega_{\eps}} \abs{D^2 g_\eps} \leq \liminf_{\eps\to0} C \int_{\Omega_{\eps}} \Delta g_\eps \leq C \hdm^{d-1}(\partial \Omega) [g]_{\contspace^{0,1}(\Omega)}.\]
Then by monotone convergence, if we take the limit $\eps \to 0$ and report this in \labelcref{bound_average_convex1}, we obtain:
\[\int_{\Omega_r} \sup_{B_{r/2}(x)} \abs{g-T_x g}\dd x \leq C \hdm^{d-1}(\partial \Omega) r^2 [g]_{\contspace^{0,1}(\Omega)},\]
for some constant $C = C(d)$.

Now we need to take care of the integral over $\Omega \setminus \Omega_r$. But if $g$ is Lipschitz on $\Omega$ with constant $L = [g]_{\contspace^{0,1}(\Omega)} < \infty$, then for almost every $x \in \Omega\setminus \Omega_r$ and every $y \in B_{r/2}(x)$, we have
\[\abs{g(y) - (g(x) + \nabla g(x) \cdot(y-x))} \leq 2L \abs{y-x} \leq L r.\]
Finally, notice that since $\Omega$ is an open convex set, by the co-area formula (see \cite[\S~3.4]{evansMeasureTheoryFine2015}), for every $r > 0$,
\[\lbm^d(\Omega\setminus \Omega_r) = \int_0^r \hdm^{d-1}(\{\dist(\cdot,\R^d\setminus \Omega) = t\}) \dd t = \int_0^r \hdm^{d-1}(\partial \Omega_t) \dd t \leq r\hdm^{d-1}(\partial \Omega),\]
and as a consequence,
\begin{align*}
\int_{\Omega} \sup_{B_{r/2}(x)} \abs{g-T_x g} &= \int_{\Omega\setminus \Omega_r} \abs{g-T_x g} + \int_{\Omega_r} \abs{g-T_x g} \\
&\leq (Lr)(r \hdm^{d-1}(\partial \Omega)) + C \hdm^{d-1}(\partial \Omega) r^2 [g]_{\contspace^{0,1}(\Omega)}\\
&= C \hdm^{d-1}(\partial \Omega)[g]_{\contspace^{0,1}(\Omega)} r^2,
\end{align*}
for some constant $C = C(d)$ and \labelcref{convexity_sup_integral_inequality} is proved.

\textbf{Step~3.} Let us conclude by establishing \labelcref{integral_lambda_convexity_Taylor_inequality}. Since $f$ is $\lambda$-convex for some $\lambda \leq 0$, we apply \labelcref{convexity_sup_integral_inequality} to the convex function $g : x \mapsto f(x) - \lambda\abs{x}^2/2$ to get
\begin{align*}
\int_{\Omega} \sup_{B_{r/2}(x)\cap \Omega} \abs{f-T_x f} &\leq \begin{multlined}[t]
\int_{\Omega} \sup_{B_{r/2}(x)\cap \Omega} \abs{g-T_x g} \dd x \\
+\frac{\abs{\lambda}}2 \int_{\Omega} \sup_{y\in B_{r/2}(x)\cap \Omega} \abs{\abs{y}^2- (\abs{x}^2+ 2x \cdot(y-x))} \dd x\end{multlined}\\
&\leq C r^2\hdm^{d-1}(\partial \Omega)[g]_{\contspace^{0,1}(\Omega)}  + \frac{\abs{\lambda}}8 r^2\lbm^d(\Omega)\\
&\leq C r^2\hdm^{d-1}(\partial \Omega)([f]_{\contspace^{0,1}(\Omega)} + \abs{\lambda}\diam(\Omega)) + \frac{\abs{\lambda}}8 r^2\lbm^d(\Omega).
\end{align*}
Knowing that $\lbm^d(\Omega) \leq \diam(\Omega)\hdm^{d-1}(\partial \Omega) $ since $\Omega$ is convex, we obtain \labelcref{lambda_convexity_Taylor_inequality} by replacing $r$ by $2r$, for some (different) constant $C = C(d)$.
\end{proof}

We are now ready to establish \Cref{value_general_upper_bound2}.

\begin{proof}[Proof of \Cref{value_general_upper_bound2}]
The measures $\mu^\pm$ are supported in some open and bounded convex sets $\Omega_0^\pm$ such that their closure $X^\pm \coloneqq \mathrm{cl}(\Omega_0^\pm)$ are included in $\Omega^\pm$. Take $\gamma_0 \in \Opt(\mu^-,\mu^+)$ and a pair of $c$-conjugate Kantorovich potentials $(\phi,\psi) \in \contspace(X^-) \times \contspace(X^+)$. Such potentials exist because $c$ is continuous on the compact set $X^-\times X^+$ (see for example \cite[Proposition~1.11]{santambrogioOptimalTransportApplied2015}). We set $\lambda \coloneqq \sup_{\mbf{x}\neq\mbf{x'} \in X^-\times X^+} \frac{\abs{\nabla c(\mbf{x'}) - \nabla c(\mbf{x})}}{\abs{\mbf{x'}-\mbf{x}}}< \infty$, so that $c$ is $\lambda$-concave on $X^-\times X^+$, which implies that $\phi,\psi$ are $\lambda$-concave on $X^-,X^+$ respectively. As such, they are differentiable Lebesgue-a.e. on $\Omega_0^-,\Omega_0^+$, thus $\mu^-$-a.e. and $\mu^+$-a.e. respectively, which in turn implies that
\[E \coloneqq c - \phi \oplus \psi\]
is differentiable $\gamma_0$-a.e. on $\Omega_0^-\times \Omega_0^+$ because $\gamma_0 \in \Pi(\mu^-,\;\mu^+)$ and $c\in \contspace^1(\Omega_0^-\times\Omega_0^+)$. Moreover, by optimality of $\gamma_0$ and $(\phi,\psi)$ we have $E \geq 0$ everywhere on $X^-\times X^+$ with equality on $\spt{\gamma_0}$. As a consequence, for $\gamma_0$-a.e. $\mbf{x_0} \in \Omega_0^-\times \Omega_0^+$, $E(\mbf{x_0}) = 0$, $E$ is differentiable at $\mbf{x_0}$ and $\nabla E(\mbf{x_0}) = 0$.

Take $\delta > 0$ and consider the block approximation $\gamma^\delta$ of $\gamma_0$ built in the proof of \Cref{value_general_upper_bound}. Since by construction $W_\infty(\gamma_0,\gamma^\delta) \leq \delta$, there exists $\kappa \in \Pi(\gamma_0,\gamma^\delta)$ such that $\norm{\mbf{x_0}-\mbf{x}} \leq \delta$ for $\kappa$-a.e. $(\mbf{x_0},\mbf{x})$, and we know in particular that $E$ is differentiable at $\mbf{x_0}$, and $T_{\mbf{x_0}} E(\mbf{x}) = 0$ for $\kappa$-a.e. $(\mbf{x_0},\mbf{x})$. Thus we have:
\begin{align*}
\int_{\Omega_0^-\times \Omega_0^+} (c-\phi \oplus \psi) \dd\gamma^\delta &= \int_{(\Omega_0^-\times \Omega_0^+)^2} (E(\mbf{x})-T_{\mbf{x_0}} E(\mbf x)) \dd\kappa(\mbf{x_0},\mbf{x}) \\
&= \begin{multlined}[t]
\int_{(\Omega_0^-\times \Omega_0^+)^2} (c(\mbf x)-T_{\mbf{x_0}} c (\mbf x))\dd\kappa(\mbf{x_0},\mbf{x})\\
- \int_{\Omega_0^-\times \Omega_0^-} (\phi(x)-T_{x_0}\phi(x)) \dd\kappa^-(x_0,x)\\
- \int_{\Omega_0^+\times\Omega_0^+} (\psi(y)-T_{y_0}\psi(y)) \dd\kappa^+(y_0,y),
\end{multlined}
\end{align*}
where $\kappa^- = (\pi_1,\pi_3)_\# \kappa, \kappa^+ = (\pi_2,\pi_4)_\# \kappa$ and $\pi_i$ denotes the projection on the $i$-th component. Since $c$ is $\lambda$-concave on $\Omega_0^-\times\Omega_0^+$,
\[c(\mbf x)-T_{\mbf{x_0}} c (\mbf x) \leq \frac{\lambda}2 \abs{\mbf{x}-\mbf{x_0}}^2,\]
and since $\phi,\psi$ are $\lambda$-concave on $\Omega_0^-, \Omega_0^+$ respectively, using \Cref{lem:alexandrov-modified} and recalling that $(\pi_1)_\# \kappa^\pm = \mu^\pm$,
\begin{align*}
\int_{\Omega_0^+\times\Omega_0^+} (c-\phi \oplus \psi) \dd\gamma^\delta &\leq \begin{multlined}[t]
\frac{\lambda}2 \abs{\mbf{x}-\mbf{x_0}}^2 + \int_{\Omega_0^-} \sup_{x\in B_\delta(x_0)\cap\Omega_0^-} \abs{\phi(x)-T_{x_0}\phi(x)} \dd \mu^-(x_0)\\
+ \int_{\Omega_0^+} \sup_{y\in B_\delta(y_0)\cap \Omega_0^+} \abs{\psi(y)-T_{y_0}\psi(y)} \dd \mu^+(y_0)
\end{multlined}\\
&\leq \begin{aligned}[t]
\lambda \delta^2 + &C \delta^2 \norm{\mu^-}_{L^\infty(\Omega_0^-)}\hdm^{d-1}(\partial \Omega_0^-)([\phi]_{\contspace^{0,1}(\Omega_0^-)}+\lambda\diam(\Omega_0^-))\\
 + &C \delta^2 \norm{\mu^+}_{L^\infty(\Omega_0^+)}\hdm^{d-1}(\partial \Omega_0^+)([\psi]_{\contspace^{0,1}(\Omega_0^+)}+\lambda\diam(\Omega_0^+))
\end{aligned}\\
&\leq C' \delta^2,
\end{align*}
for some constant $C'>0$ which does not depend on $\delta$.

Now, we apply \labelcref{upper_bound_box_counting_number} to assert that $H_\delta(\mu^+) \leq \log(N_\delta(X^+)) \leq d\log(1/\delta) + \log C$ for some constant $C > 0$ depending only on $X^+$, and we use $\gamma^\delta$ as a competitor in \labelcref{eot_pb} to obtain:
\begin{align*}
v_\eps \leq \int_{\Omega^-\times \Omega^+} c \dd\gamma^\delta + \eps H_\delta(\mu^+) &= v_0 + \int_{\Omega^-\times \Omega^+} (c-\phi\oplus \psi) \dd\gamma^\delta + \eps H_{\delta}(\mu^+)\\
&\leq  v_0 + C' \delta^2 + \eps H_\delta(\mu^+)\\
&= v_0 + C'\delta^2 + d\eps\log(1/\delta)  + \eps \log C.
\end{align*}
Finally, taking $\delta = \sqrt\eps$ yields:
\[v_\eps \leq v_0 + \frac d2 \eps\log(1/\eps) + \eps(C'+\log C).\qedhere\]
\end{proof}

%stability as corollary

\section{Lower bound for infinitesimally twisted costs}\label{sec:lower}

In this section, we establish a lower bound that is analogous to the fine upper bound given in \Cref{sec:fine_upper_bound}, for a class of $\contspace^2$ costs satisfying an \emph{infinitesimal twist condition}, which includes situations where optimal plans are not necessarily given by a map. Our proof is based on the so-called Minty's trick asserting that the graph of a monotone operator is a $1$-Lipschitz graph in a rotated chart (see \cite{mintyMonotoneNonlinearOperators1962}). It has been used in particular by Alberti and Ambrosio \cite{albertiGeometricalApproachMonotone1999} to study the fine properties of monotone functions on $\R^d$, and more recently in optimal transport by McCann, Pass and Warren \cite{mccannRectifiabilityOptimalTransportation2012} to show the rectifiability of optimal transport plans. We will exploit this trick a little bit further to show a quadratic detachment of the duality gap, leading to the lower bound that we seek. As a side remark, we deduce from the same trick a quantitative stability result in the quadratic case $c(x,y) = \abs{x-y}^2$.

\subsection{Proof of the lower bound}

We start with the definition of the infinitesimal twist condition, corresponding to condition (A2) in \cite{maRegularityPotentialFunctions2005}, and the non-degeneracy condition in \cite{mccannRectifiabilityOptimalTransportation2012}.

\begin{definition}\label{def:inf-twist}
Given $c \in \contspace^2(\Omega^- \times \Omega^+)$ where $\Omega^\pm\subseteq \R^d$ are open sets, we say that $c$ is \emph{infinitesimally twisted} if $\nabla^2_{xy} c(x,y) \coloneqq (\partial^2_{x_i y_j} c(x,y))_{i,j} \in M_d(\R)$ is invertible for every $(x,y) \in \Omega^- \times \Omega^+$.
\end{definition}

McCann, Pass and Warren have proved that for such a cost, the support of any optimal transport plan is locally Lipschitz (see \cite[Theorem~1.2]{mccannRectifiabilityOptimalTransportation2012}).

% \begin{theorem}[McCann, Pass and Warren, Theorem~1.2 \cite{mccannRectifiabilityOptimalTransportation2012}]\label{thm_mccann}
% Let $c \in \contspace^2(\Omega \times \Omega)$ be infinitesimally twisted, and $\mu^\pm \in \probspace(\R^d)$ be compactly supported in $\Omega$. If $\gamma \in \Pi(\mu^-,\;\mu^+)$ is optimal for $c$ then $\spt{\gamma}$ is included in a $d$-dimensional Lipschitz submanifold of $\R^d$, and is therefore $d$-rectifiable.
% \end{theorem}
% \pap{C'est ce qu'ils prétendent, mais est-ce vraiment complètement clair à partir de leur preuve qu'il y a un lipschiz manifold global contenant le support ? Localement c'est vrai, mais il faut recoller globalement...}

We closely follow the computations of \cite{mccannRectifiabilityOptimalTransportation2012} leading to the proof of their main theorem, but we consider points which do not necessarily belong to the support of an optimal plan.

\begin{lemma}\label{gap_lemma}
Let $c \in \contspace^2(\Omega^- \times \Omega^+)$ be an infinitesimally twisted cost, and $(\phi,\psi) \in \contspace(X^-)\times \contspace(X^+)$ be a pair of $c$-conjugate functions on compact convex sets $X^\pm\subseteq \Omega^\pm$. We set $\ener \coloneqq c- \phi \oplus \psi$ on $X^-\times X^+$, $\Sigma \coloneqq \{ \ener = 0\}$, and for every $r > 0$,
\[\kappa(r)  \coloneqq \sup_{\substack{(\mbf{x},\mbf{x'}) \in X^-\times X^+\\\norm{\mbf{x'}-\mbf{x}}\leq r}} \norm{\nabla^2_{xy} c(\mbf{x'})^{-1} \nabla^2_{xy} c(\mbf{x})-\id} \in [0,\infty).\]
If $\mbf{\bar x} \in X^-\times X^+$ and $\mbf{x},\mbf{x'} \in B_r(\mbf{\bar x}) \cap (X^-\times X^+)$, then
\begin{equation}\label{gap_inequality}
\ener(\mbf{x'}) + \ener(\mbf{x}) \geq \abs{\Delta u}^2 - \abs{\Delta v}^2 - \kappa(r) (\abs{\Delta u}^2 + \abs{\Delta v}^2),
\end{equation}
where we have set $\Delta u \coloneqq u(\mbf{x'})-u(\mbf{x})$, $\Delta v \coloneqq v(\mbf{x'})-v(\mbf{x})$, and
\[u(\mbf{x}) \coloneqq \frac 12 (x+\nabla^2_{xy} c(\mbf{\bar x}) y), \quad v(\mbf{x}) \coloneqq \frac 12 (x-\nabla^2_{xy} c(\mbf{\bar x}) y), \quad \text{for every } \mbf{x} = (x,y).\]
\end{lemma}
\begin{proof}
Take $\mbf{\bar x} \in X^-\times X^+$, $\mbf{x} = (x,y)$ and $\mbf{x'} = (x',y')$ in $B_r(\mbf{\bar x}) \cap (X^-\times X^+)$. Knowing that $(\phi,\psi)$ is a pair of $c$-conjugate functions, and using Taylor's integral formula,
\begin{align*}
\ener(\mbf{x'}) &= c(x',y') - \phi(x') - \psi(y')\\
&\geq c(x',y') - (c(x',y)-\psi(y)) - (c(x,y')-\phi(x))\\
&= c(x',y') - c(x',y) - c(x,y') + c(x,y) - \ener(\mbf{x})\\
&= -\ener(\mbf{x}) + (x'-x) \cdot \left[ \int_0^1 \int_0^1 \nabla^2_{xy} c(x+(1-s)x',y+(1-t)y') \dd s \dd t\right] (y'-y).
\end{align*}
With the notations introduced in the statement of the lemma, $x = u(\mbf{x})+v(\mbf{x})$ and $y = \nabla^2_{xy} c(\mbf{\bar x})^{-1}(u(\mbf{x})-v(\mbf{x}))$, and similarly for $\mbf{x'}$, so that:
\begin{align*}
\MoveEqLeft \ener(\mbf{x'}) + \ener(\mbf{x})\\
&\geq \begin{multlined}[t]
(\Delta u + \Delta v)\: \cdot\\
\left[\int_0^1 \int_0^1 \nabla^2_{xy} c(x+(1-s)x',y+(1-t)y') \nabla^2_{xy} c(\mbf{\bar x})^{-1}\dd s \dd t \right](\Delta u - \Delta v)
\end{multlined}\\
%&\geq (\Delta u + \Delta v)\cdot(\Delta u - \Delta v) - \kappa(r)\abs{\Delta u + \Delta v}\abs{\Delta u - \Delta v}\\
&\geq \abs{\Delta u}^2-\abs{\Delta v}^2 - \kappa(r)\abs{\Delta u + \Delta v}\abs{\Delta u - \Delta v}\\
&\geq \abs{\Delta u}^2-\abs{\Delta v}^2 - \kappa(r)(\abs{\Delta u}^2 + \abs{\Delta v}^2),
\end{align*}
the last inequality resulting from the fact that $\abs{a+b}\abs{a-b} \leq \abs{a}^2+\abs{b}^2$ for every $a,b\in \R^d$, as can be seen by expanding and comparing the squares of the two sides.
\end{proof}
\begin{remark}
The rectifiability theorem of McCann, Pass and Warren \cite[Theorem~1.2]{mccannRectifiabilityOptimalTransportation2012} is an immediate consequence of \Cref{gap_lemma} (to no surprise, since we have essentially followed their computations). Indeed, for any optimal plan $\gamma \in \Pi(\mu^-,\;\mu^+)$ with $\mu^\pm$ supported on compact convex sets $X^\pm \subseteq \Omega^\pm$, by taking $\mbf{\bar x}, \mbf{x}, \mbf{x'} \in \spt{\gamma} \subseteq \Sigma$ , \labelcref{gap_inequality} yields
\[\abs{\Delta u} \leq \sqrt{\frac{1+\kappa(r)}{1-\kappa(r)}} \abs{\Delta v},\]
and since $\kappa(r) \xto{r\to 0} 0$, $v$ is a Lipschitz function of $u$, so that $B_r(\mbf{\bar x}) \cap \spt{\gamma}$ is included in a Lipschitz $d$-dimensional graph. 

Notice however that \labelcref{gap_inequality} gives extra information: if we take for example $\mbf{x} = \mbf{\bar x}$, and we fix $v' = v = \bar v$ but we let the $u$ component free, we get
\[\ener(\mbf{x'}) \geq (1-\kappa(r)) \abs{u'-u}^2.\]
Thus we have a \emph{quadratic growth} of $E$ away from $\mbf{x}$ in the direction given by the $u$ coordinate. This is what we shall use to get the lower bound.
\end{remark}

\begin{proposition}\label{prop:lower-bound}
Let $\Omega^\pm$ be open convex subsets of $\R^d$, $c \in \contspace^2(\Omega^- \times \Omega^+)$ be an infinitesimally twisted cost, and $\mu^\pm \in \probspace(\Omega^\pm)\cap L^\infty(\Omega^\pm)$ with compact support in $\Omega^\pm$. There exists a constant $m\in [0,\infty)$ such that for every $\eps >0$,
\begin{equation}\label{lower_bound}
v_\eps \geq v_0 + \frac d2 \eps\log(1/\eps) - m \eps.
\end{equation}
\end{proposition}
\begin{proof}
The measures $\mu^\pm$ being concentrated on some compact convex subsets $X^\pm \subseteq \Omega^\pm$, consider a pair $(\phi,\psi) \in \contspace(X^-)\times\contspace(X^+)$ of $c$-conjugate Kantorovich potentials. Taking $(\phi,\psi)$ as competitor in \labelcref{dual_eot_pb}, we get the lower bound:
\begin{align*}
v_\eps &\geq \int_{X^-} \phi \dd\mu^- + \int_{X^+} \psi\dd\mu^+ -\eps \log\left(\int_{X^-\times X^+} e^{-\frac{E}\eps} \dd\mu^-\otimes\mu^+\right)\\
&= v_0  -\eps \log\left(\int_{X^-\times X^+} e^{-\frac{E}\eps} \dd\mu^-\otimes\mu^+\right),
\end{align*}
where $E \coloneqq c-\phi\oplus \psi$ on $X^- \times X^+$ as in \Cref{gap_lemma}. We are going to show that for some constant $C>0$ and for every $\eps >0$,
\[\int_{X^-\times X^+} e^{-E/\eps} \dd \mu^-\otimes\mu^+ \leq C \eps^{d/2},\]
which yields \labelcref{lower_bound} with $m = \log(C)$.

We follow the notation introduced in the statement of \Cref{gap_lemma}, and consider in particular the functions $u,v$ defined for a fixed $\mbf{\bar x} \in X^-\times X^+$. We define the affine map $\Phi_{\mbf{\bar x}} : \mbf{x} \mapsto (u(\mbf{x}),v(\mbf{x}))-(u(\mbf{\bar x}),v(\mbf{\bar x}))$ whose Jacobian determinant is easily computed: $J\Phi_{\mbf{\bar x}} = (-1/2)^d \det (\nabla^2_{xy} c(\mbf{\bar x}))$. Thus $\Phi_{\mbf{\bar x}}$ is an affine isomorphism and $\abs{J(\Phi_{\bar x})^{-1}} \leq C_1 \coloneqq 2^d \sup_{\mbf{x} \in X^-\times X^+} \abs{\det \nabla^2_{xy} c(\mbf{x})}^{-1} < \infty$ because $c$ is infinitesimally twisted. We set $E_{\mbf{\bar x}} \coloneqq E \circ \Phi_{\mbf{\bar x}}^{-1}$ over $D_{\mbf{\bar x}} \coloneqq \Phi_{\mbf{\bar x}}(X^-\times X^+)$ and define the open neighborhood of $\mbf{\bar x}$, $P_{\mbf{\bar x}} \coloneqq \Phi_{\mbf{\bar x}}^{-1} (B_r \times B_r) \subseteq B_{Lr}(\mbf{\bar x})$, where $L \coloneqq 2(1\wedge \sup_{\mbf{x} \in X^-\times X^+} \norm{(\nabla_{xy}^2 c(\mbf{x}))^{-1}})$ and $r > 0$ is chosen such that $\kappa(L r) \leq 1/2$. By the change of variable formula we get:
\begin{align*}
\MoveEqLeft \int_{P_{\mbf{\bar x}}\cap X^-\times X^+} e^{-E/\eps} \dd \mu^- \otimes \mu^+\\
&= \int_{(B_r \times B_r) \cap D_{\mbf{\bar x}}} e^{-E(\Phi_{\mbf{\bar x}}^{-1}(u,v))/\eps} (\mu^- \otimes \mu^+)(\Phi_{\mbf{\bar x}}^{-1}(u,v)) \abs{J(\Phi_{\mbf{\bar x}})^{-1}(u,v)} \dd u \dd v\\
&\leq C_1 \norm{\mu^-}_{L^\infty(\Omega^-)} \norm{\mu^+}_{L^\infty(\Omega^+)} \int_{\pi_2((B_r\times B_r) \cap D_{\mbf{\bar x}})} \int_{\{u \in B_r : (u,v) \in D_{\mbf{\bar x}}\}} e^{-E_{\mbf{\bar x}}(u,v)/\eps} \dd u \dd v.
\end{align*}
Now, for every $v \in \pi_2((B_r\times B_r) \cap D_{\mbf{\bar x}})$, consider $u_v$ a minimizer of $E_{\mbf{\bar x}}(\cdot,v)$ over $\{u \in \bar B_r : (u,v) \in {D_{\mbf{\bar x}}}\}$. By \labelcref{gap_inequality} of \Cref{gap_lemma}, for every $u \in B_r$ such that $(u,v) \in D_{\mbf{\bar x}}$,
\[E_{\mbf{\bar x}}(u,v) \geq \frac{1}{2} (E_{\mbf{\bar x}}(u_v,v)+ E_{\mbf{\bar x}}(u,v))) \geq \frac{1}{2}(1-\kappa(Lr)) \abs{u-u_v}^2  \geq \frac{1}{4} \abs{u-u_v}^2 .\]
As a consequence we obtain:
\begin{align*}
\int_{\pi_2((B_r\times B_r) \cap D_{\mbf{\bar x}})} \int_{\{u \in B_r : (u,v) \in D_{\mbf{\bar x}}\}} e^{-E_{\mbf{\bar x}}(u,v)/\eps} \dd u \dd v &\leq  \int_{\pi_2(B_r\times B_r) \cap D_{\mbf{\bar x}}} \int_{B_r} e^{-\abs{u-u_v}^2/{2\eps}}\dd u \dd v\\
&\leq \omega_d r^d \eps^{d/2} \int_{\R^d} e^{-\abs{u}^2/2} \dd u\\
&\leq C_2 r^d \eps^{d/2},
\end{align*}
for some constant $C_2 > 0$.
The sets $\{P_{\mbf{\bar x}}\}_{\mbf{\bar x}\in \Sigma}$ form an open covering of the compact set $\Sigma \subseteq X^-\times X^+$, hence we may extract a finite covering $P_{\mbf{\bar x_1}}, \ldots, P_{\mbf{\bar x_N}}$, so that $\Sigma \subseteq \bigcup_{i=1}^N P_{\mbf{\bar x_i}}$ and:
\[\int_{\bigcup_{i=1}^N P_{\mbf{\bar x_i}}\cap X^-\times X^+} e^{-E/\eps} \mu^- \otimes \mu^+ \leq N C_1 C_2\norm{\mu^-}_{L^\infty(\Omega^-)} \norm{\mu^+}_{L^\infty(\Omega^+)} r^d \eps^{d/2} \leq C_3\eps^{d/2},\]
for some constant $C_3 > 0$. Finally, since $E$ is continuous and does not vanish on the compact set $K = (X^-\times X^+) \setminus \bigcup_{i=1}^N P_{\mbf{\bar x_i}}$, it is bounded from below on $K$ by some constant $C_4 > 0$. Therefore, for every $\eps > 0$,
\[\int_{X^-\times X^+} e^{-E/\eps} \dd \mu^- \otimes\mu^+ \leq C_3\eps^{d/2} + e^{-C_4/\eps} \leq  C \eps^{d/2},\]
for some constant $C > 0$. This concludes the proof.
\end{proof}

\subsection{Quantitative stability of optimal plans for the quadratic cost}

This final paragraph is devoted to the quadratic-cost case for which combining Minty's trick and the upper bound on the entropic cost, one can obtain a quantitative estimate between the optimal entropic plan and optimal transport plans. Given two probability measures $\mu^{\pm} \in \probspace(\R^d)$ with compact support, taking as cost function $c(x,y) \coloneqq \frac 12 \abs{x-y}^2$, the transport cost is merely the square $2$-Wasserstein distance:
\[v_0 =\frac{1}{2} W_2^2(\mu^-, \mu^+) \coloneqq \frac{1}{2} \inf_{\gamma \in \Pi(\mu^-,\;\mu^+)} \int_{\R^d\times \R^d} \abs{x-y}^2 \dd \gamma(x,y).\]
It is well known (see Brenier \cite{brenierPolarFactorizationMonotone1991}, McCann \cite{mccannExistenceUniquenessMonotone1995}) that there exists a convex lsc function $f$ on $\R^d$ such that $\gamma \in \Pi(\mu^-,\;\mu^+)$ is optimal in the above quadratic OT problem if and only if 
\[\spt{\gamma} \subseteq \Gamma_f\coloneqq \{(x,y)\in \R^d \times \R^d \; : \;  f(x)+f^*(y)=x\cdot y\}=\{(x,y)\in \R^d \times \R^d \; : \; y\in \partial f(x)\}\]
where $f^*$ is the Legendre transform of $f$. Kantorovich potentials are then given by 
\[\phi(x) \coloneqq \frac{1}{2} \abs{x}^2-f(x), \quad  \psi(y)\coloneqq\frac{1}{2} \abs{y}^2-f^*(y), \qquad (\forall x,y\in \R^d).\]
If $\mu^-$ is absolutely continuous with respect to the Lebesgue measure, $f$ is differentiable $\mu^-$-a.e. and there is a unique optimal plan $\gamma\coloneqq(\id, T)_\# \mu^-$ where $T\coloneqq\nabla f$ is Brenier's optimal transport map from $\mu^-$ to $\mu^+$. For $\eps>0$, the entropic cost reads
\[v_\eps = \inf_{\gamma \in \Pi(\mu^-,\;\mu^+)} \Bigg\{ \frac{1}{2} \int_{\R^d\times \R^d} \abs{x-y}^2 \dd \gamma(x,y) +\eps\entropy(\gamma\,|\,\mu^- \otimes \mu^+)\Bigg\},\]
and we denote as before by $\gamma_\eps$ the optimal entropic plan. Our goal is to give an estimate on how $\gamma_\eps$ fails to be concentrated on $\Gamma_f$ for small $\eps>0$ in a sense to be made precise. First we observe that by nonnegativity of the entropic term, we have
\[v_\eps-v_0 \geq \frac{1}{2} \int_{\R^d\times \R^d} \abs{x-y}^2 \dd \gamma_\eps(x,y)-\frac{1}{2} W_2^2(\mu^-, \mu^+).\]
Defining the duality gap as in the previous section by:
\[E(x,y)\coloneqq \frac{1}{2} \abs{x-y}^2-\phi(x)-\psi(y)=f(x)+f^*(y)-x \cdot y \qquad (\forall x,y \in \R^d),\]
one has $E \geq 0$ and $E(x,y)=0$ if and only if $(x,y)\in \Gamma_f$, i.e. $y\in \partial f(x)$. Denoting by $\gamma_0$ an optimal transport plan, since $E=0$ on $\spt{\gamma_0}$ and $\gamma_\eps$ and $\gamma$ share the same marginals, we have 
\begin{equation}\label{contvE1}
v_\eps-v_0 \geq \frac{1}{2} \int_{\R^d\times \R^d} \abs{x-y}^2 \dd \gamma_\eps(x,y)-\frac{1}{2} W_2^2(\mu^-, \mu^+)= \int_{\R^d\times \R^d} E(x,y) \dd \gamma_\eps(x,y).
\end{equation}
In the event where $\nabla f$ is $M$-Lipschitz, which, by Caffarelli's regularity theory \cite{caffarelliBoundaryRegularityMaps1992,caffarelliBoundaryRegularityMaps1996} holds true if $\mu^\pm$ have H\"older densities bounded away from zero on their supports and the latter are smooth and uniformly convex, arguing as Berman in \cite{bermanConvergenceRatesDiscretized2021} and as Li and Nochetto in \cite{liQuantitativeStabilityError2021}, who build upon an earlier argument of Gigli \cite{gigliHolderContinuityintimeOptimal2011}, one can use the inequality
\begin{equation}\label{linoclikegap}
E(x,y)=f^*(y)-f^*(\nabla f(x))-x \cdot (y-\nabla f(x)) \geq \frac{1}{2M} \abs{y-\nabla f(x)}^2
\end{equation}
(using the fact that $x\in \partial f^*(\nabla f(x))$ and $f^{*}-\frac{1}{2M} \abs{\cdot}^2$ is convex as soon as $\nabla f$ is $L$-Lipschitz) to arrive at:
\begin{proposition}
If Brenier's optimal transport map $T=\nabla f$ between the compactly supported probability measures $\mu^-$ and $\mu^+$ is $M$-Lipschitz, denoting by $\gamma_\eps$  the optimal entropic plan between $\mu^-$ and $\mu^+$, one has 
\begin{equation}\label{entropiclnerror}
\int_{\R^d\times \R^d} \abs{y-T(x)}^2 \dd \gamma_\eps(x,y) \leq M(d \eps \log(1/\eps) +O(\eps)).
\end{equation}
In particular, if $T_\eps$ denotes the barycentric projection of $\gamma_\eps$ (i.e. $T_\eps(x)$ is the conditional expectation of $Y$ given $X=x$ with $(X,Y)$ distributed according to $\gamma_\eps$), there holds
\begin{equation}\label{barycerror}
\norm{T_\eps- T}^2_{L^2(\mu^-) } \leq M(d \eps \log(1/\eps) +O(\eps)).
\end{equation}
\end{proposition}

\begin{proof}
Inequality \labelcref{entropiclnerror} follows directly from \labelcref{contvE1}, \labelcref{linoclikegap} and the upper bound $v_\eps-v_0 \leq d \eps \log(1/\eps) +O(\eps)$ from  \Cref{value_general_upper_bound} (note also that if $\mu^\pm$ have bounded densities, one can improve the factor $d$ by $\frac{d}{2}$ thanks to   \Cref{value_general_upper_bound2}). Finally \labelcref{barycerror} directly follows from \labelcref{entropiclnerror} and Jensen's inequality:
\begin{align*}
\int_{\R^d\times \R^d} \abs{y-T(x)}^2 \dd \gamma_\eps(x,y) &= \int_{\R^d} \int_{\R^d} \abs{y-T(x)}^2 \dd \gamma_\eps^x(y) \dd \mu^-(x)\\
&\geq  \int_{\R^d}  \abs*{ \int_{\R^d} (y-T(x)) \dd \gamma_\eps^x(y)}^2  \dd \mu^-(x)\\
&= \int_{\R^d} \abs{T_\eps(x)-T(x)}^2\dd\mu^-(x),
\end{align*}
where $\gamma_\eps^x$ is the disintegration of $\gamma_\eps$ with respect to its first marginal.
\end{proof}

Of course the requirement that the convex potential $f$ is smooth is quite demanding and cannot be taken for granted in general.  Yet, if $f$ is an arbitrary lsc convex function, one can take advantage of Minty's trick to have a quadratic detachment lower bound on $E$ as we did in \Cref{gap_lemma}.

\begin{proposition}
If $\mu^{\pm} \in \probspace(\R^d)$ are compactly supported probability measures, and if we denote by $\gamma_\eps$ the optimal entropic plan from $\mu^-$ to $\mu^+$, the following stability bound holds:
\[ \int_{\R^d\times \R^d} \abs{x-(\id + \partial f)^{-1}(x+y)}^2 \dd \gamma_\eps(x,y) \leq (d \eps \log(1/\eps) +O(\eps)),\]
where $(\id +\partial f)^{-1}$ is the (single-valued) resolvent associated with $f$.
\end{proposition}
\begin{proof}
Let us first observe that  when $c$ is the quadratic cost \labelcref{gap_inequality} takes the form
\[E(x,y) \geq -E(x',y') - (x'-x) \cdot (y'-y) \qquad (\forall x,y,x', y' \in \R^d).\]
Now we observe that the resolvent $(\id +\partial f)^{-1}$ is a single-valued $1$-Lipschitz map and that $E(x',y')$ vanishes if and only if $ y'+x' \in x' + \partial f(x')$, hence if and only if $x'= (\id +\partial f)^{-1}(x'+y')$. As a consequence if we choose in the inequality above $x'= (\id +\partial f)^{-1}(x+y)$ and $y'=x+y-x'$, we have $E(x',y') = 0$ and we get the quadratic detachment bound:
\begin{equation}\label{quaddetach}
E(x,y) \geq \abs{x-(\id+\partial f)^{-1}(x+y)}^2, \qquad (\forall x,y \in \R^d).
\end{equation}
Hence, from \labelcref{quaddetach}, \labelcref{contvE1} and  \Cref{value_general_upper_bound}, we obtain the desired inequality (and again one can improve the factor $d$ by $\frac{d}{2}$ thanks to  \Cref{value_general_upper_bound2} if $\mu^\pm$ have bounded densities). 
\end{proof}

\smallskip

\noindent{\textbf{Acknowledgments:}} G.C.  acknowledges the support of the Lagrange Mathematics and Computing Research Center.

%%%%%%%%%%%%%%%%%%%%%%%%%%%%%%%%%%%%%%%%%%%%%%%%%%%%%%
{\sloppy%\emergencystretch=1em
\printbibliography
}

\end{document}